\documentclass[oneside,american,english]{amsart}
\usepackage[T1]{fontenc}
\usepackage[latin9]{inputenc}
\usepackage{amstext}
\usepackage{amsthm}
\usepackage{amssymb}
\usepackage{stmaryrd}
\usepackage[all]{xy}

\makeatletter
\numberwithin{equation}{section}
\numberwithin{figure}{section}
\theoremstyle{plain}
\newtheorem{thm}{\protect\theoremname}
\theoremstyle{plain}
\newtheorem{prop}[thm]{\protect\propositionname}
\theoremstyle{plain}
\newtheorem{lem}[thm]{\protect\lemmaname}
\theoremstyle{plain}
\newtheorem{cor}[thm]{\protect\corollaryname}
\theoremstyle{plain}
\newtheorem{conjecture}[thm]{\protect\conjecturename}

\usepackage{tikz}
\usetikzlibrary{cd}
\newlength{\perspective}
\usepackage{amsmath,amsthm,amssymb}
\usepackage{mathtools}
\usepackage{MnSymbol}

\makeatother

\usepackage{babel}
\addto\captionsamerican{\renewcommand{\conjecturename}{Conjecture}}
\addto\captionsamerican{\renewcommand{\corollaryname}{Corollary}}
\addto\captionsamerican{\renewcommand{\lemmaname}{Lemma}}
\addto\captionsamerican{\renewcommand{\propositionname}{Proposition}}
\addto\captionsamerican{\renewcommand{\theoremname}{Theorem}}
\addto\captionsenglish{\renewcommand{\conjecturename}{Conjecture}}
\addto\captionsenglish{\renewcommand{\corollaryname}{Corollary}}
\addto\captionsenglish{\renewcommand{\lemmaname}{Lemma}}
\addto\captionsenglish{\renewcommand{\propositionname}{Proposition}}
\addto\captionsenglish{\renewcommand{\theoremname}{Theorem}}
\providecommand{\conjecturename}{Conjecture}
\providecommand{\corollaryname}{Corollary}
\providecommand{\lemmaname}{Lemma}
\providecommand{\propositionname}{Proposition}
\providecommand{\theoremname}{Theorem}

\begin{document}
\selectlanguage{american}%
\global\long\def\half{\frac{1}{2}}%

\selectlanguage{english}%
\global\long\def\OO{\mathcal{O}}%

\global\long\def\DD{\mathcal{D}}%

\global\long\def\KK{\mathcal{K}}%

\global\long\def\EE{\mathcal{E}}%

\global\long\def\eered{\mathcal{E}_{red}^{\bullet}}%

\global\long\def\fred{F_{red}^{\bullet}}%

\global\long\def\EEE{\mathbb{E}}%

\global\long\def\FF{\mathcal{\mathcal{F}}}%

\global\long\def\FFF{\mathbb{F}}%

\global\long\def\GG{\mathcal{G}}%

\global\long\def\ss{\mathbb{S}}%

\global\long\def\III{\mathcal{I}}%

\global\long\def\LL{\mathcal{L}}%

\global\long\def\XY{X_{-y}}%

\global\long\def\XX{\mathcal{X}}%

\global\long\def\PPP{\mathcal{P}}%

\global\long\def\PXB{\mathcal{P}_{\chi}\left(X,\beta\right)}%

\global\long\def\PX{\mathcal{P}_{\chi}(X,i_{*}\beta)}%

\global\long\def\PXX{\mathcal{P}_{\chi}\left(\bar{X},\bar{i}_{*}\beta\right)}%

\global\long\def\PS{\mathcal{P}_{\chi}(S,\beta)}%

\global\long\def\PSL{\mathcal{P}_{\chi}(S,|\LL|)}%

\global\long\def\CC{\mathcal{C}}%

\global\long\def\ZZZ{\mathcal{Z}}%

\global\long\def\CCC{\mathbb{C}}%

\global\long\def\CS{\mathbb{C}^{\times}}%

\global\long\def\SN{S^{[n]}}%

\global\long\def\AA{\mathbb{A}}%

\global\long\def\BB{\mathbb{B}}%

\global\long\def\KKK{\mathbb{K}}%

\global\long\def\XXX{\mathbb{X}}%

\global\long\def\HH{\mathbb{H}}%

\global\long\def\LLL{\mathbb{L}}%

\global\long\def\PP{\mathbb{P}}%

\global\long\def\II{\mathbb{I}^{\bullet}}%

\global\long\def\barII{\bar{\mathbb{I}}^{\bullet}}%

\global\long\def\CS{\mathbb{C}^{\times}}%

\global\long\def\QQ{\mathbb{Q}}%

\global\long\def\ZZ{\mathbb{Z}}%

\global\long\def\tweight{\mathfrak{t}}%

\global\long\def\tweighth{\mathfrak{t}^{1/2}}%

\global\long\def\tweightc{\left(-\mathfrak{t}^{-\frac{1}{2}}\right)}%

\global\long\def\tweighta{1-\mathfrak{t}}%

\global\long\def\tweightb{\mathfrak{t}^{-1/2}-\mathfrak{t}^{1/2}}%

\global\long\def\taus{\gamma\left(\mathcal{O}_{s}\right)}%

\global\long\def\tausi{\gamma\left(\OO_{s_{i}}\right)}%

\global\long\def\tauss{\hat{\gamma}\left(\OO_{s}\right)}%

\global\long\def\taussi{\hat{\gamma}\left(\OO_{s_{i}}\right)}%

\global\long\def\tausss{\bar{\gamma}(\OO_{s})}%

\global\long\def\tausssi{\bar{\gamma}(\OO_{s_{i}})}%

\global\long\def\ya{1-y}%

\global\long\def\yb{y^{-1/2}-y^{1/2}}%

\global\long\def\TM{TS^{[n]}\times\mathbb{P}^{\chi(\LL)-1}}%

\global\long\def\ee{\mathcal{L}^{[n]}\boxtimes\mathcal{O}(1)}%

\global\long\def\id{\text{id}}%

\global\long\def\Pxlmn{P_{\chi,\beta}^{red}(X,\left[pt\right]^{m})}%

\global\long\def\Pslmn{P_{\chi,\beta}^{red}(S,\left[pt\right]^{m})}%

\global\long\def\Pxlm{P^{red}(X,\left[pt\right]^{m}}%

\global\long\def\pxlmn{P_{X,\beta,\chi}(s_{1},\ldots,s_{m})}%

\global\long\def\bpxlmn{\bar{P}_{X,\beta,\chi}\left(s_{1},\ldots,s_{m}\right)}%

\global\long\def\pxlm{P_{X,\beta,m}}%

\global\long\def\pslmn{P_{S,\mathcal{L},m,\chi}}%

\global\long\def\bpslmn{\left(\frac{1}{y}\right)^{\frac{d-m}{2}}\pslmn}%

\global\long\def\bpslm{P_{S,\LL,m}}%

\global\long\def\barchi{\bar{\chi}}%

\global\long\def\res{\text{res}}%

\global\long\def\rk{\text{rk}}%

\global\long\def\coeff{\text{Coeff}}%

\global\long\def\hilb{\text{Hilb}}%

\global\long\def\pic{\text{Pic}}%

\global\long\def\sym{\text{Sym}}%

\global\long\def\HOM{\mathcal{H}om}%

\global\long\def\Ext{\text{Ext}}%

\global\long\def\hom{\text{Hom}}%

\global\long\def\hhom{\mathcal{H}om}%

\global\long\def\tor{\mathcal{T}or}%

\global\long\def\ext{\mathcal{E}xt}%

\global\long\def\proj{\text{Proj}}%

\global\long\def\ch{\text{ch}}%

\global\long\def\td{\text{td}}%

\global\long\def\codim{\text{codim}}%

\global\long\def\div{\text{div}}%

\renewcommand{\div}{\text{div}}

\global\long\def\cycle{\text{cycle}}%

\global\long\def\spec{\text{Spec}}%

\global\long\def\supp{\text{Supp}}%

\global\long\def\d{h-1+n+\int_{\beta}c_{1}(S)+h^{0,2}(S)}%

\global\long\def\Lim{{\displaystyle \lim_{\leftarrow}}}%

\title{A Refinement of Kool-Thomas Invariants via Equivariant K-theoretic
invariants}
\author{Rizal Afgani}
\address{Analysis and Geometry Group, Faculty of Mathematics and Natural Sciences,
Bandung Institute of Technology, Bandung 40132, INDONESIA}
\email{rafgani@math.itb.ac.id}
\begin{abstract}
In this article we are defining a refinemement of Kool-Thomas invariants
of local surfaces via the equivariant $K$-theoretic invariants proposed
by Nekrasov and Okounkov. Kool and Thomas defined the reduced obstruction
theory for the moduli of stable pairs $\PX$ as the degree of the
virtual class $\left[\PS\right]^{red}$ afted we apply $\tau([pt])^{m}\in H^{*}(\PX,\ZZ)$.
$\tau([pt])$ contain the information of the incidence of a point
and a curve supporting a $(\FF,s$).

The $K$-theoretic invariants proposed by Nekrasov and Okounkov is
the equivariant holomorphic Euler characteristic of $\OO_{\PX}^{vir}\otimes K_{vir}^{\half}$.
We introduce two classes $\taus$ and $\tausss$ in the Grothendieck
group of vector bundles on the moduli space of stable pairs of the
local surfaces that contains the information of the incidence of a
curve with a point. We define two invariants $\pxlmn$ and $\bpxlmn$
each corresponds to $\taus$ and $\tausss$. We found that the contribution
of $\PS\subset\PPP^{G}$ to $\pxlmn$ and to $\bpxlmn$ are the same.
Moreover, if we evaluate this contribution at $\tweight=1$ we get
the Kool-Thomas invariants.

The generating function of this contribution contain the same information
as the generating function of the refined curve counting invariants
defined by G\"ottsche and Shende in \cite{GS:14}. After a change
of variable there exist a coefficient $N_{\delta[S,\LL]}^{\delta}(y)$
of the generating function of the refined curve counting that counts
the number of $\delta$-nodal curve in $\PP^{\delta}\subset|\LL|$.
We conjecture that after the same change of variable the corresponding
coefficient $M_{\delta[S,\LL]}^{\delta}(y)$ coming from the generating
function of the contribution of $\PS$ to $\bpxlmn$ is identical
with $N_{\delta[S,\LL]}^{\delta}(y).$

\end{abstract}

\maketitle

\section{Introduction}

Fix a nonsingular projective surface $S$ and a sufficently ample
line bundle $\LL$ on $S$. A $\delta$-nodal curve $C$ on $S$ is
a $1$ dimensional subvariety of $S$ which has nodes at $\delta$
points and is regular outside these singular points. For any scheme
$Y$, let $Y^{[n]}$ be the Hilbert scheme of $n$-points i.e. $Y^{[n]}$
parametrizes subschemes $Z\subset Y$ of length $n$. Given a family
of curves $\CC\rightarrow B$ over a base $B$, we denote by $\hilb^{n}(\CC/B)$
the relative Hilbert scheme of points. Kool, Thomas and Shende showed
that some linear combinations $n_{r,C}$ of the Euler characteristic
of $C^{[n]}$ counts the number of curves of arithmetic genus $r$
mapping to $C$. Applying this to the family $\CC\rightarrow\PP^{\delta}$
where $\PP^{\delta}\subset|\LL|$, the number of $\delta$-nodal curves
is given by a coefficient of the generating function of the Euler
characteristic of $\hilb(\CC/\PP^{\delta})$ after change of variable\cite{KST:11}.
By replacing euler characteristic with Hirzebruch $\chi_{y}$-genus,
G\"otsche and Shende give a refined counting of $\delta$-nodal curves.

Pandharipande and Thomas showed that a stable pair $(\FF,s)$ on a
surface $S$ is equivalent to the pair $(C,Z)$ of a curve $C$ on
$S$ supporting the sheaf $\FF$ with $Z\subset C$ a subscheme of
finite length. Thus the moduli space of stable pairs on a surface
$S$ is a relative Hilbert scheme of points corresponding to a family
of curves on $S$.

The study of the moduli space of stable pairs on Calabi-Yau threefold
$Y$ is an active area of research. This moduli space gives a compactification
of the moduli space of nonsingular curves in $Y$. To get an invariant
of the moduli space Behrend and Fantechi introduce the notion of perfect
obstruction theory. With this notion we can construct a class in the
Chow group of dimension 0 that is invariant under some deformations
of $Y$\cite{BF:97}.

The homological invariants of the stable pair moduli space $\PX$
of the total space $X$ of $K_{S}$ of some smooth projective surface
$S$ contain the information of the number of $\delta$-nodal curves
in a hyperplane $\PP^{\delta}\subset|\LL|$. Notice that $X$ is Calabi-Yau.
There exist a morphism of schemes $\div:\PX\rightarrow|\LL|$ that
maps a point $(\FF,s)\in\PX$ to a divisor $\div\left(\pi_{*}\FF\right)$
that support $\pi_{*}\FF$ on $S$ where $\pi:X\rightarrow S$ is
the structure morphism of $X$ as a vector bundle over $S$. Using
descendents, Kool and Thomas translate the information of the incidence
of a curve with a point into cutting down the moduli space by a hypersurface
pulledback from $|\LL|$ so that after cutting down, we have a moduli
space that parameterize Hilbert scheme of curves in $\PP^{\delta}$\cite{KT:14}.

The famous conjecture of Maulik, Nekrasov Okounkov and Pandharipande
states that the invariants corresponding to the moduli space of stable
pairs have the same information as the invariants defined from the
moduli space of stable maps and the Hilbert schemes.

The next development in the theory of PT invariants is to give a refinement
of the homological invariant. The end product of this homological
invariant is a number. A refinement of this invariant would be a Laurent
polynomial in a variable $t$ such that when we evaluate $t$ at $1$
we get the homological invariant.

There are several methods that have been introduced to give a refinement
for DT invariants, for example both motivic and $K$-theoretic definitions.
In this thesis we use the $K$-theoretic definition which has been
proposed by Nekrasov and Okounkov in \cite{NO:14} where we compute
the holomorphic Euler characteristic of the twisted virtual structure
sheaf of the coresponding moduli space. In the case when $S=\PP^{2}$
or $S=\PP^{1}\times\PP^{1}$ Choi, Katz and Klemm have computed a
$K$-theoretic invariant of the moduli space of stable pairs in the
paper \cite{CKK:12}. Their computation does not include any information
about the incidence of subschemes of $S$.

\section{Equivariant Chow Groups and $K$-theory\label{sec:Equivariant--theory}}

In this section we will describe the notatioan and definition we use
regarding equivariant Chow Groups and $K$-theory.

\subsection{Equivariant Chow Groups}

In this section we review the definition of equivariant Chow groups
given in \cite{EG:98,EG:00}. We will use $g$ to denote the dimension
of our group $G$ as a scheme over $\CCC$.

Given $i\in\ZZ.$ Let $X$ be a $G$-scheme with $\dim\,X=d$. Let
$V$ be $G$-vector space of dimension $l$. Assume that there exists
an open subscheme $U\subset V$ and a principal $G$-bundle $\pi:U\rightarrow U_{G}.$By
giving $X\times V$ a diagonal action of $G$, assume furthermore
that there exist a principal $G$-bundle $\pi_{X}:X\times U\rightarrow\left(X\times U\right)/G$.
We will use $X\times_{G}U$ to denote $\left(X\times U\right)/G$.
Assume also that $V\setminus U$ has codimension greater than $d-i$,
then the equivariant Chow group is defined as 
\[
A_{i}^{G}(X):=A_{i+l-g}(X\times_{G}U).
\]
The definition is independent up to isomorphism of the choice of a
representation as long as $V\setminus U$ is of codimension greater
than $d-i$ .

For a $G$-equivariant map $f:X\rightarrow Y$ with property $P$
where $P$ is either proper, flat, smooth, or regular embedding the
$G$-equivariant map $f\times1:X\times U\rightarrow Y\times U$ has
the property $P$ since all of these properties are preserved by a
flat base change. Moreover, the corresponding morphism $f_{G}:X\times_{G}U\rightarrow Y\times_{G}U$
also has property $P$. In fact, these properties are local on the
target in the Zariski topology and for any trivialization $(V_{i},\bar{\varphi}_{i})_{i\in\Lambda}$
of $\pi:U\rightarrow U_{G}$ the restriction of $f_{G}$ on $\pi_{X}(X\times\pi^{-1}(V_{i}))$
is isomorphic to $f\times\id_{V_{i}}$. So from the definition, for
a flat $G$ -map $f:X\rightarrow Y$ of codimension $l$ we can define
pullback map $f^{*}:A_{i}^{G}(Y)\rightarrow A_{i+l}^{G}(X)$ for equivariant
Chow groups. Similarly, for regular embedding $f:X\rightarrow Y$
of codimension $d$ we have a Gysin homomorphism $f^{*}:A_{i}^{G}(Y)\rightarrow A_{i-d}^{G}(X)$
and for proper $G$-map $f:X\rightarrow Y$ we can define pushforward
$f_{*}:A_{i}^{G}(X)\rightarrow A_{i}^{G}(Y)$ for equivariant Chow
groups.

For $G=T_{1}$ and an $l+1$-dimensional weight space $V_{\chi}$
we have a principal $G$-bundle $\pi_{U}:=V_{\chi}\setminus\{0\}\rightarrow\PP(V_{\chi})$.
There exist a principal $G$-bundle $\pi_{X}:X\times U\rightarrow X\times_{G}U$.
And since $\codim\,V_{\chi}\setminus U$ is $l+1$, for each $i\in\text{\ensuremath{\ZZ\ }}$
we can take $A_{i+l}\left(X\times_{G}U)\right)$ to represent $A_{i}^{G}(X)$
if $l+i\geq d$. We can also fix $\chi$ to be $-1$ to cover all
$i$. 

Thus we fix the following notation. For each positive integer $l$
let $V_{l}$ be a $T_{1}$-space of weight $-1$ with coordinate $x_{0},\ldots,x_{l}$.
Thus $V_{l-1}$ is the zero locus of the last coordinate of $V_{l}$.
We use $U_{l}$ to denote $V_{l}\setminus\{0\}$ and $X_{l}$ to denote
$X\times_{G}U_{l}$ and $\pi_{X,l}:X\times U_{l}\rightarrow X_{l}$
the corresponding principal bundle. Thus we have the following direct
system 
\begin{equation}
\xymatrix{\ldots\ar[r] & X_{l-1}\ar[r]^{j_{X,l-1}} & X_{l}\ar[r]^{j_{X,l}} & X_{l+1}\ar[r]^{j_{X,l+1}} & \ldots}
\label{direct-system}
\end{equation}

\selectlanguage{american}%
There is a projection from $\xi:V_{l+1}\rightarrow V_{l}$ by forgetting
the last coordinate such that $j_{l}:V_{l}\rightarrow V_{l+1}$ is
the zero section of $\xi.$ By removing the fiber of $p:=(0:0:\ldots:0:1)\in\PP(V_{l+1})$,
the corresponding projection $\xi:X_{l+1}\setminus\pi_{X}^{-1}(p)\rightarrow X_{l}$
is a line bundle over $X_{l}$ such that $j_{X,l}:X_{l}\rightarrow X_{l+1}\setminus\pi_{X,l+1}^{-1}(p)$
is the zero section. Note that $\dim\text{\ensuremath{\pi_{X,l+1}^{-1}(x)}=\ensuremath{\dim\,X=d}}$.
Thus for $i\geq d-l$ the restriction map $A_{i+l+1}(X_{l+1})\rightarrow A_{i+l+1}(X_{l+1}\setminus\pi_{X,l+1}^{-1}(p))$
is an isomorphism. In general this restriction is a surjection. Since
$\hat{j}_{X,l}:X_{l}\rightarrow X_{l+1}\setminus\pi_{X,l+1}^{-1}(p)$
is the zero section of $\xi$ , the Gysin homomorphism $\hat{j}_{X,n}^{!}:A_{k+1}(X_{l+1}\setminus\pi_{X,l+1}^{-1}(p))\rightarrow A_{k}(X_{l})$
is an isomorphism. Since $j$ is a regular embedding we have a Gysin
homomorphism $j^{!}:A_{k+1}\left(X_{l+1}\right)\rightarrow A_{k}(X_{l})$
which is the composition of the above homomorphisms.

\sloppy The direct system \ref{direct-system} induces an inverse
system 
\[
\xymatrix{\ldots & A_{*}(X_{l-1})\ar[l] & A_{*}(X_{l})\ar[l]_{j_{X,l-1}^{!}} & A_{*}(X_{l+1})\ar[l]_{j_{X,l}^{!}} & \ldots}
\]
 of abelian groups. Let $(\Lim A(X_{l}),\lambda_{l})$ be the inverse
limit of the above inverse system. \foreignlanguage{english}{From
the definition of equivariant Chow groups, $A_{i}^{G}(X)=A_{i+n}(X_{n})$
for $i\geq d-n$ so that we can identify $\prod_{i=d-n}^{d}A_{i}^{G}(X)$
with the group $\prod_{i=d}^{d+n}A_{i}(X_{n})$. Recall that $\left({\displaystyle \prod_{i=-\infty}^{d}}A_{i}^{G}(X),\nu_{i}\right)$
where $\nu_{n}:$${\displaystyle \prod_{i=-\infty}^{d}}A_{i}^{G}(X)\rightarrow\prod_{i=d-n}^{d}A^{G}(X)$
is defined by $(a_{d},a_{d-1}\ldots)\mapsto\left(a_{d},\ldots,a_{d-n}\right)$
is the inverse limit of the inverse system defined by the projection
$p_{X,n}:$$\prod_{i=d-n-1}^{d}A^{G}(X)\rightarrow\prod_{i=d-n}^{d}A^{G}(X)$,
$(a_{d},\ldots,a_{d-n},a_{d-n-1})\mapsto(a_{d},\ldots,a_{d-n})$.
}After indentifying $\prod_{i=d-n}^{d}A_{i}^{G}(X)$ with \foreignlanguage{english}{$\prod_{i=d}^{d+n}A_{i}(X_{n})$,
$p_{X,n}$ and $j_{X,n}^{!}$ are the same homomorphism. The compostion
of the projections $\hat{\xi}_{n}:$$A_{*}(X_{n})\rightarrow\prod_{i=d}^{d+n}A(X_{n})$
with $\lambda_{n}:$$\Lim A_{*}(X_{l})\rightarrow A_{*}(X_{n})$ are
homorphisms $\xi_{i}:$$\Lim A_{*}(X)\rightarrow\prod_{i=d-n}^{d}A_{i}^{G}(X)$
satisfying $p_{X,n+1}\circ\xi_{n}=p_{X,n}$ so that by the universal
property of inverse limit we have a group homomorphism $\xi:$$\Lim A_{*}(X_{n})\rightarrow{\displaystyle \prod_{i=-\infty}^{d}}A_{i}^{G}(X)$
satisfying $p_{X,n}\circ\xi=\xi_{n}$. We will use the following proposition
later.}
\begin{prop}
\label{lim-isom}$\xi:\Lim A_{*}(X_{l})\rightarrow{\displaystyle \prod_{i=-\infty}^{d}}A_{i}^{G}(X)$
is an isomorphism .
\end{prop}

\selectlanguage{english}%

\subsection{Equivariant $K$-theory}

Let $G$ acts on a scheme $X$. Let $Vec_{G}(X)$ (resp. $Coh_{G}(X)$)
be the category of equivariant vector bundle (resp. equivariant coherent
sheaves) on scheme $X$. We will use $K^{G}(X)$ (resp. $G^{G}(X)$)
to denote the grup generated by equivariant vector bundle (resp. coherent
sheaves) modulo short exact sequence. More generally for any exact
category $\mathcal{N}$ we can construct the group $K_{0}\mathcal{N}$
generated by objects of $\mathcal{N}$ modulo short exact sequence.

\selectlanguage{american}%
We will skecth the construction of pushforward map $f_{*}:K^{G}(X)\rightarrow K^{G}(Y)$
in some special cases induced by direct image functor. For more details,
readers should consult chapter 2 of \cite{We:13} or section 7 and
8 of \cite{Qu:72}. Since taking pullback is exact on vector bundles,
the definition of $f^{*}:K^{G}(Y)\rightarrow K^{G}(X)$ is straightforward.

First we need the following Lemma.
\begin{lem}
\label{lem4}Let $\mathcal{N}_{X}$ be a full subcategory of $Coh_{G}(X)$
staisfying the following conditions:

1. $\mathcal{N}_{X}$ contains $Vec_{G}(X)$ 

2. $\mathcal{N}_{X}$ is closed under extension

3. Each objects of $\mathcal{N}_{X}$ has a resolution by a bounded
complex of elements in $Vec_{G}(X)$ 

4. $\mathcal{N}_{X}$ is closed under kernels of surjections. 

Then 

1. $\mathcal{N}_{X}$ is exact and the inclusion $Vec_{G}(X)\subset\mathcal{N}_{X}$
induce the group homomorphism $i:K^{G}(X)\rightarrow K_{0}\left(\mathcal{N}_{X}\right)$
by mapping the class $\left[\PPP\right]_{Vec_{G}(X)}$ of any locally
free sheaf $\PPP$ to its class $\left[\PPP\right]_{\mathcal{N}_{X}}$
in $K_{0}(\mathcal{N}_{x})$ 

2. all resolutions of $\FF$ by equivariant locally free sheaves
\[
\xymatrix{0\ar[r] & \PPP_{n}\ar[r] & \PPP_{n-1}\ar[r] & \ldots\ar[r] & \PPP_{1}\ar[r] & \PPP_{0}\ar[r] & \FF\ar[r] & 0}
\]

define the same element $\chi(\FF):=\sum_{i=0}^{n}\left(-1\right)^{-i}\left[\PPP_{i}\right]$
in $K^{G}(X)$. Furthermore, $\chi$ define a group homomorphism $\chi:K_{0}(\mathcal{N}_{X})\rightarrow K^{G}(X)$
which is the inverse of $i:K^{G}(X)\rightarrow K_{0}(\mathcal{N}_{X})$.
\end{lem}

\begin{cor}
\label{Cor1}Let $f:X\rightarrow Y$ be a finite $G$-morphism such
that $f_{*}:Vec_{G}(X)\rightarrow Coh_{G}(X)$ factors through a subcatcategory
$\mathcal{N}_{Y}\subset Coh_{G}(Y)$ satisfying all 4 conditions of
lemma \ref{lem4} above. Then there exist a group homomorphism $f_{*}:K^{G}(X)\rightarrow K^{G}(Y)$
such that $f_{*}[\EE]=\chi(f_{*}\EE)$ for any locally free sheaf
$\EE$ on $X$.
\end{cor}

\begin{prop}
[Projection Formula] \label{prop:=00005BProjection-Formula=00005D}Let
$f:X\rightarrow Y$ be a morphism satisfying the condition in corollary
\ref{Cor1} or the projection $\varphi:\PP_{Y}(V)\rightarrow Y$ where
$V$ is a $G$-equivariant vector bundle. Then for any $x\in K^{G}(X)$
and $y\in K^{G}(Y)$ we have 
\[
f_{*}\left(x.f^{*}y\right)=\left(f_{*}x\right).y\in K^{G}(Y).
\]
\end{prop}

\begin{prop}
[Base change formula]\label{prop:BaseChange}
\end{prop}

\selectlanguage{english}%
\begin{enumerate}
\item \textit{Consider the following cartesian diagram
\[
\xymatrix{\bar{X}\ar[r]^{\bar{g}}\ar[d]_{\bar{f}} & X\ar[d]^{f}\\
\bar{Y}\ar[r]_{g} & X
}
\]
such that $f$ and $f'$ are $G$-regular embeddings of codimension
$r$. Then $g^{*}\circ f_{*}=\bar{f}_{*}\circ\bar{g}^{*}:K^{G}(X)\rightarrow K^{G}(\bar{Y})$}
\selectlanguage{american}%
\item \textit{Let $A$ be a smooth projective variety and let $p:A\times Y\rightarrow Y$
be the projection to the second factor. Let $g:\bar{Y}\rightarrow Y$
be any morphism and consider the following cartesian diagram 
\[
\xymatrix{A\times\bar{Y}\ar[d]_{\bar{p}}\ar[r]^{\bar{g}} & A\times Y\ar[d]^{p}\\
\bar{Y}\ar[r]_{g} & Y
}
.
\]
Then the pushforward maps $p_{*}:K^{G}(A\times Y)\rightarrow K^{G}\left(Y\right)$
and $\bar{p}_{*}:K^{G}(A\times\bar{Y})\rightarrow K^{G}\left(\bar{Y}\right)$
are well defined and $\bar{p}_{*}\circ\bar{g}^{*}=g^{*}\circ p_{*}:K^{G}(A\times Y)\rightarrow K^{G}\left(\bar{Y}\right)$.
Let $d:D\rightarrow A\times Y$ be a $G$-closed embedding such that
$D$ is flat over $Y$ and let $d':D'\rightarrow A\times Y'$ be the
corresponding pullback so that we have the following cartesian diagram
\[
\xymatrix{\bar{D}\ar[r]^{\hat{g}}\ar[d]_{\bar{d}} & D\ar[d]^{d}\\
A\times\bar{Y}\ar[r]_{\bar{g}} & A\times Y.
}
\]
Then $\bar{g}^{*}\left[\OO_{D}\right]=\left[\OO_{\bar{D}}\right]\in K^{G}(A\times\bar{Y})$.}
\end{enumerate}
\selectlanguage{american}%
Let $i:X\rightarrow Y$ be a $G$-equivariant closed embdedding and
let $U=Y\setminus X$ with open embedding $j:U\rightarrow Y$. Then
there exist group homomorphism $i_{*}:G^{G}(X)\rightarrow G^{G}(Y)$
and $j^{*}:G^{G}(Y)\rightarrow G^{G}(U)$. These two homomorphism
is related as follows:
\begin{lem}
\label{lem:support1}The following complex of abelian groups is exact
\[
\xymatrix{G^{G}(X)\ar[r]^{i_{*}} & G^{G}(Y)\ar[r]^{j^{*}} & G^{G}(U)\ar[r] & 0}
.
\]
\end{lem}

\begin{proof}
This is Theorem 2.7 of \cite{Th:83}.
\end{proof}
We call a class $\beta\in G^{G}(Y)$ is supported on $X$ if $\beta$
is in the image of $i_{*}$. Equivalently $\beta$ is supported on
$X$ if $j^{*}\beta=0$.

Let $Coh_{G}^{X}(Y)$ be the abelian group of coherent sheaves supported
on $X$. Note that $\FF\in Coh_{G}^{X}(Y)$ is not necessarily an
$\OO_{X}$-module. Let $G_{X}^{G}(Y)$ be the corresponding Grothendieck
group. The pushforward functor $i_{*}:Coh_{G}(X)\rightarrow Coh_{G}(Y)$
factors through $Coh_{G}^{X}(Y)$ so that there exist a group homomorphism
$\bar{i}:G^{G}(X)\rightarrow G_{X}^{G}(Y)$, $[\FF]\mapsto[i_{*}\FF]$.
There exist an inverse of $\bar{i}$ described as follows.

Let $\FF\in Coh_{G}^{X}(Y)$ and let $\III$ be the ideal of $X.$
Then there exist positive integer $n$ such that $\III^{n}\FF=0$
so that we have a filtration 
\[
\FF\supseteq\III\FF\supseteq\III^{2}\FF\supseteq\ldots\supseteq\III^{n-1}\FF\supseteq\III^{n}\FF=0.
\]
Note that each $\III^{r}\FF/\III^{r+1}\FF$ is an $\OO_{X}$-module.
One can show that $[\FF]\mapsto\sum_{r=0}^{n-1}[\III^{r}\FF/\III^{r+1}\FF]$
defines a group homomorphism $\bar{i}^{-1}:G_{X}^{G}(Y)\rightarrow G^{G}(X)$. 
\begin{lem}
\label{lem:support}$\bar{i}:G^{G}(X)\rightarrow G_{X}^{G}(Y)$ is
an isomorphism.
\end{lem}

\selectlanguage{english}%
Given a cartesian diagram 
\[
\xymatrix{\bar{X}\ar[r]^{\bar{f}}\ar[d]_{\bar{i}} & \bar{Y}\ar[d]^{i}\\
X\ar[r]_{f} & Y
}
\]
with $i$, $f$ are closed embeddings and a coherent sheaf $\EE$
on $X$ such that $f_{*}\EE$ has a finite resolution by a complex
of locally free sheaves. Then we can define a group homomorphism $f^{[\EE]}:G^{G}(\bar{Y})\rightarrow G^{G}(\bar{X}),$
described as follows. Let $\FF$ be a coherent sheaf on $Y$ supported
on $\bar{Y}$. For each $y\in Y$ , the stalk of $\tor_{Y}^{i}(f_{*}\EE,\FF)$
on $y$ is $\tor_{\OO_{Y,y}}^{i}\left(\left(f_{*}\EE\right)_{y},\FF_{y}\right)$
so that $\tor_{y}^{i}\left(f_{*}\EE,\FF\right)$ is supported on $\bar{X}$.
For any exact sequence 
\[
\xymatrix@C=8pt{0\ar[r] & \FF'\ar[r] & \FF\ar[r] & \FF"\ar[r] & 0}
\]
of coherent sheaves on $\bar{Y}$ we have a long exact sequence
\[
\xymatrix@C=8pt{\tor_{Y}^{i+1}(f_{*}\EE,\FF")\ar[r] & \tor_{Y}^{i}(f_{*}\EE,\FF')\ar[r] & \tor_{Y}^{i}(f_{*}\EE,\FF)\ar[r] & \tor_{Y}^{i}(f_{*}\EE,\FF')\ar[r] & \,}
\]
so that 
\[
\sum_{i\geq0}(-1)^{i}[\tor_{Y}^{i}(f_{*}\EE,\FF)]=\sum_{i\geq0}(-1)^{i}[\tor_{Y}^{i}(f_{*}\EE,\FF')+\sum_{i\geq0}(-1)^{i}[\tor_{Y}^{i}(f_{*}\EE,\FF")]\in G_{\bar{X}}^{G}(Y).
\]
Thus there exist a group homomorphism $\bar{f}^{[\EE]}:G^{G}(\bar{Y})\rightarrow G_{\bar{X}}^{G}(Y)$.
By Lemma \ref{lem:support}, we can define $f^{[\EE]}$ as the composition
$\bar{i}^{-1}\circ\bar{f}^{[\EE]}$.

\begin{lem}
\label{lem:support0}Let $f:X\rightarrow Y$ be a closed embedding
and a coherent sheaf $\EE$ on $X$ such that $f_{*}\EE$ has a finite
resolution by locally free sheaves. For any closed embedding $i:\bar{Y}\rightarrow Y$,
there exist a group homomorphism $f^{[\EE]}:G^{G}(\bar{Y})\rightarrow G^{G}(\bar{Y}\cap X)$
that maps $[\FF]$ to $\sum_{i=0}(-1)^{-1}\left[\tor_{Y}^{i}(f_{*}\EE,\FF)\right]_{\bar{Y}\cap X}$.
Furthermore, $i_{*}f^{[\EE]}(\left[\FF\right])=\sum_{i=0}(-1)^{-1}\left[\tor_{Y}^{i}(f_{*}\EE,\FF)\right]_{Y}$.
\end{lem}

Let $G$ be the torus $T_{1}$ and let $X$ be a $G$-scheme. Recall
that by Proposition \ref{lim-isom} there exist an isomorphism $\xi:\Lim A_{*}(X_{n})\rightarrow\prod_{i=-\infty}^{d}A_{i}^{G}(X)$.
In this section we want to recall some results of the corresponding
$\Lim K(X_{n}).$

\selectlanguage{american}%
From the direct system \ref{direct-system}, we have the inverse system
\[
\xymatrix{\ldots\ar[r] & K(X_{l-1})\ar[l] & K(X_{l})\ar[l]_{j_{X,l-1}^{*}} & K(X_{l+1})\ar[l]_{j_{X,l}^{*}} & \ldots}
\]
We denote the inverse limit of the above inverse system as $\Lim\,K(X_{l})$
and use $\rho_{X,l}$ to denote the canonical morphism $\Lim K(X_{l})\rightarrow K(X_{l})$.
The pullback functor induced from the projection map $pr_{X}:X\times U_{l}\rightarrow X$
and the equivalence between $Vec_{G}(X\times U_{l})$ and $Vec(X_{l})$
induces group homomorphims $\kappa_{X,l}:K^{G}(X)\rightarrow K(X_{l})$.
It's easy to show that $\kappa_{X,l}=j_{X,l}^{*}\circ\kappa_{X,l+1}$
so that we have a uniqe group homomorphism $\kappa_{X}:K^{G}(X)\rightarrow\Lim K(X_{l})$
such that $\kappa_{X,l}=\rho_{X,l}\circ\kappa_{X}$. In this section,
to distinguish bertween the ordinary and the equivariant version of
pullback and pushforward map, we will use superscript $^{G}$ to denote
the equivariant version, for example we will use $f^{G,*}$ to denote
the pullback in the equivariant setting.

There is a canonical way to define pullback map $\overleftarrow{f^{*}}:\Lim K(Y)\rightarrow\Lim K(X)$
and pushforward map $\overleftarrow{f_{*}}:\Lim K(X_{l})\rightarrow\Lim K(Y_{l})$
for $\Lim\,K(X_{l})$.
\selectlanguage{english}%
\begin{lem}
\label{lem:kappa}Let $f:X\rightarrow X$ be a $G$ morphism.

1. If $f:X\rightarrow Y$ is a finite $G$-morphism satisfying the
condition in corrolary \ref{Cor1}. Assume also that for all $l$
, $\left(f\times\id_{U_{l}}\right)$ also satisfies the condition
in corrolary \ref{Cor1} . Then there exist a group homomorphism $\overleftarrow{f_{*}}:\Lim K(X_{l})\rightarrow\Lim K(Y_{l})$
satisfying the identity $\kappa_{Y}\circ f_{*}^{G}=\overleftarrow{f_{*}}\circ\kappa_{X}$.

2. If $f:X\rightarrow Y$ is the structure morphism $\PP_{Y}(V)\rightarrow Y$
where $V$ is a $G$-equivariant vector bundle. Then there exist a
group homomorphism $\overleftarrow{f}_{*}:\Lim K(X_{l})\rightarrow\Lim K(Y_{l})$
satisfying the identity $\kappa_{Y}\circ f_{*}^{G}=\overleftarrow{f_{*}}\circ\kappa_{X}$.

3. If $f:X\rightarrow Y$ is a $G$-morphism that can be factorized
into $p\circ i$ where $i:X\rightarrow Z$ is a finite morphism satisfying
the condition 1. and $p$ satisfies condition 2. then the group homomorphsim
$\overleftarrow{f_{*}}:=\overleftarrow{p_{*}}\circ\overleftarrow{i_{*}}:$$\Lim K(X)\rightarrow\Lim K(Y)$
is independent of the factorization.
\end{lem}

For each equivariant vector bundle $\EE$ on $X$, its pullback $\tilde{\EE}$
to $X\times U_{n}$ correspond to a vector bundle $\EE_{n}$ on $X_{n}$
such that $\pi^{*}\EE_{n}=\tilde{\EE}$. By the identification $A_{j}^{G}(X)=A_{j+n}(X_{n})$,
$c_{G}^{i}(\EE):A_{j}^{G}(X)\rightarrow A_{j-i}^{G}(X)$ is given
by $c^{i}(\EE_{n}):A_{j+n}(X_{n})\rightarrow A_{j-i+n}(X_{n})$. Since
Chern class commutes with pullback this definition is well defined.
Furthermore, $c_{G}^{j}(\EE)$ is an element of $A_{G}^{i}(X)$.

In the non equivariant case, each vector bundle $\EE$ of rank $r$
has Chern roots $x_{1},\ldots,x_{r}$ such that $c^{i}(\EE)=e_{i}(x_{1},\ldots,x_{r})$
where $e_{i}$ is the $i^{\text{th}}$ symmetric polynomial. Furthermore,
its Chern character is defined as $ch(\EE)=\sum_{i=1}^{r}e^{x_{i}}$.
From this definition, we have the following formula of Chern chararacter
in terms of Chern classes
\begin{align*}
ch(\EE) & =r+c^{1}(\EE)+\frac{1}{2}\left(c^{1}(\EE)^{2}-2c^{2}(\EE)\right)+..\\
 & =\sum_{i=0}^{\infty}P_{j}(c^{1}(\EE),\ldots,c^{i}(\EE))
\end{align*}
 where $P_{j}\left(c^{1}(\EE),\ldots,c^{j}(\EE)\right)$ is a polynomial
of order $j$ with $c^{i}(\EE)$ has weight $i$.

In \cite{EG:00}, Edidin and Graham define an equivariant Chern character
map $ch^{G}:K^{G}(X)\rightarrow\prod_{i=0}^{\infty}A_{G}^{i}(X)$
by the following formula 
\[
ch^{G}\left(\EE\right)=\sum_{i=0}^{\infty}P_{i}(c_{G}^{1}(\EE),\ldots,c_{G}^{i}(\EE)).
\]
One can show that $ch^{G}$ is a ring homomorphism. Let $\overleftarrow{ch}:K^{G}(X)\rightarrow\Lim A^{*}(X_{n})$
denote the composition $\alpha\circ ch^{G}$. 

For each $n$ there is a Chern character map $ch_{n}:K(X_{n})\rightarrow A^{*}(X_{n})$
which commutes with refined Gysin homomorphisms. By the universal
property of inverse limits we have a ring homomorphism $\widehat{ch}:\Lim K(X_{n})\rightarrow\Lim A^{*}(X_{n})$.
Since each $ch_{n}$ is a ring homomorphis, $\widehat{ch}$ is also
a ring homomorphism. Furthermore the following diagram commutes\begin{equation}
\begin{tikzcd}
K^G(X)\ar[r,"ch^G"]\ar[d,"\kappa"']&\displaystyle{\prod_{i=0}^{\infty}}A^i _G(X)\ar[d,"\alpha"']\\
\Lim K(X_n)\ar[r,"\widehat{ch}"']&\Lim A^*(X_n)
\end{tikzcd}
\end{equation}
\begin{lem}
\label{chern}For all $x\in\Lim A_{*}(X_{n})$ and for any $\beta\in K^{G}(X)$
we have $\xi\left(\overleftarrow{ch}(\beta)(x)\right)$$=ch^{G}(\beta)(\xi x)$.
\end{lem}

\section{Kool-Thomas invariants \label{sec:Kool-Thomas-invariants}}

In this section we will review the definition of Kool-Thomas invariants
and its relation to curve counting.

Let $X$ be a projective smooth varietiy of dimension 3. A pair $\left(\FF,s\right)$
where $\FF$ is a coherent sheaf of dimension $1$ and $s$ is a section
of $\FF$ is called stable if the following two conditions holds:
\begin{enumerate}
\item $\FF$ is pure
\item The cokernel $Q$ of $s$ is of dimension 0.
\end{enumerate}
Let $X$ be a smooth projective 3-fold and let $\chi$ be an interger
and $\beta$ be a class in $H_{2}(X,\ZZ)$, there exists a projective
scheme $\PXB$ parametrizing pairs $\left(\FF,s\right)$ satisfying
the above conditions with scheme theoretic support of $\FF$ is of
class $\beta$ and holomorphic Euler characteristic $\chi(\FF)$ equals
to $\chi$\cite{Po:93}. We wiil use $\PPP$ to denote $\PXB$ whenever
the context is clear. There exist a universal sheaf $\FFF$ and universal
section $\ss:\OO_{\PPP\times X}\rightarrow\FFF$ on the product space
$X\times\PPP.$ We denote by $p$ and $q$ the projection from $\PXB\times X$
to the factor $\PXB$ and $X$ respectively. Note that in general
$\PX$ is singular.

If $G$ acts on $X$ there is a natural $G$ action on $\PPP$. Moreover,
if $G$ acts diagonally on $\PPP\times X$ i.e. $\sigma_{\PPP\times X}:G\times\PPP\times X\rightarrow\PPP\times X$,
$(g,p,x)\mapsto(g.p,g.x)$ , then the universal sheaf $\FFF$ is an
equivariant sheaf and $\ss:\OO_{\PPP\times X}\rightarrow\FFF$ is
an equivariant morphism of sheaves. 

Given a map $\phi:E^{\bullet}\rightarrow\LLL_{Y}$, where $E^{\bullet}=\{E^{-1}\rightarrow E^{0}\}$
is a two term complex of vector bundles and $\LLL_{Y}$ is the cotangent
complex of $Y$, such that the induced map $\phi^{0}$ , $\phi^{-1}$
is isomorphism and epimorphism respectively, Behrend and Fantechi
construct a class $[Y]^{vir}\in A_{\rk E^{0}-\rk E^{-1}}(Y)$ called
virtual fundamental class\cite{BF:97}. The map $\phi$ is called
a perfect obstruction theory of $Y$ and $vd:=\rk E^{0}-\rk E^{-1}$
is the virtual dimension of $Y$. The perfect obstruction theory gives
a virtual class even though the space $Y$ is badly singular.

Let $\II=\{\OO_{\PPP\times X}\rightarrow\FFF\}$ be the universal
complex on $\PPP\times X$. Pandharipande and Thomas have shown that
$Rp_{*}\left(R\hhom\left(\II,\II\right)_{0}\otimes\omega_{X}\right)[2]$
is a two term complex of locally free sheaves and there exist a map
$\phi:Rp_{*}\left(R\hhom\left(\II,\II\right)_{0}\otimes\omega_{X}\right)[2]\rightarrow\LLL_{\PPP}$
satisfying the above conditions. We will use $\EEE^{\bullet}=\{E^{-1}\rightarrow E^{0}\}$
to denote the complex $Rp_{*}\left(R\hhom\left(\II,\II\right)_{0}\otimes\omega_{p}\right)[2]$
on $\PPP$. The virtual dimension of $\PXB$ is then $-\chi(R\hhom\left(I^{\bullet},I^{\bullet}\right)_{0})=\int_{\beta}c_{1}(X)$.
If $X$ is Calabi-Yau the dualizing sheaf $\omega_{X}\simeq\OO_{X}$
so that by Serre duality $vd=\rk E^{0}-\rk E^{-1}=0$. If $vd=0$
then $P_{X,\beta,\chi}:=\int_{\left[\PPP\right]^{vir}}1\in\ZZ$ and
is invariant along a deformation of $X$. $P_{X,\beta,\chi}$ is called
Pandharipande-Thomas invariant or PT-invariant. 

One technique to compute PT-invariants is using the virtual localization
formula by Graber and Pandharipande. If $G=\CS$ acts on $\PXB$ then
$\LLL_{\PXB}$ has a natural equivariant structure. Let $\PPP^{G}$
be the fixed locus of $\PPP$, then $\EEE^{\bullet}$ has a sub-bundle
$\left(\left.\EEE^{\bullet}\right|_{\PPP^{G}}\right)^{fix}$ which
has weight 0 and a sub-bundle $\left(\left.\EEE^{\bullet}\right|_{\PPP^{G}}\right)^{mov}$
with non zero weight such that $\left.\EEE^{\bullet}\right|_{\PPP^{G}}=\left(\left.\EEE^{\bullet}\right|_{\PPP^{G}}\right)^{fix}\oplus\left(\left.\EEE^{\bullet}\right|_{\PPP^{G}}\right)^{mov}$.
Graber and Pandharipande showed that there exists a canonical morphism
$\hat{\phi}:\left(\left.\EEE^{\bullet}\right|_{\PPP^{G}}\right)^{fix}\rightarrow\LLL_{\PPP^{G}}$
that induces a perfect obstruction theory for $\PPP^{G}$. So that
we have the virtual fundamental class $\left[\text{\ensuremath{\PPP}}^{G}\right]^{vir}$
of $\PPP^{G}$. Graber and Pandaripandhe gives a formula that relates
$[\PPP^{G}]^{vir}$ with $\left[\PPP\right]^{vir}$ as follows :
\[
\left[\PPP\right]^{vir}=i_{*}\left(\frac{[\PPP^{G}]^{vir}}{e(N^{vir})}\right)\in A_{*}^{G}\otimes_{\ZZ}\QQ[t,t^{-1}]
\]
 where $e\left(N^{vir}\right)$ is the top Chern class of the vector
bundle $N^{vir}=\left(\left(\left.\EEE^{\bullet}\right|_{\PPP^{G}}\right)^{mov}\right)^{\vee}$
and $t$ is the first Chern class of the equivariant line bundle with
weight $1$.

Let $S$ be a nonsingular projective surface with canonical bundle
$\omega_{S}$ and let $X$ be the total space of $\omega_{S}$ i.e.
$X=Spec\left(\sym^ {}(\omega_{S}^{\vee})\right)$. Then there is a
closed embedding $i$ of $S$ into $X$ as the zero section. Let $\pi:X\rightarrow S$
be the structure morphism. Since $\omega_{X}\simeq\pi^{*}\omega_{S}\otimes\pi^{*}\omega_{S}^{\vee}\simeq\OO_{X}$,
$X$ is Calabi-Yau. Let $\bar{X}=\PP(X\oplus\AA_{S}^{1})$, then $X$
is an open subscheme of $\bar{X}$ and let $j:X\rightarrow\bar{X}$
be the inclusion and $\bar{\pi}:\bar{X}\rightarrow S$ be the structure
morphism of $\bar{X}$ as a projective bundle over $S$. Since $S$
is projective, $\bar{i}:=j\circ i:S\rightarrow\bar{X}$ is a closed
embedding.

Let $\beta\in H_{2}(S,\ZZ)$ be an effective class and $\chi\in\ZZ$.
By \cite{PT:09} there is a projective scheme $\PXX$ parametrizing
stable pairs $(\FF,s)$ with $\chi(\FF)=\chi$ and the cycle $[C_{\FF}]$
of the supporting curve is in class $\beta$. By removing the pairs
$(\FF,s)$ with supporting curve $C_{\FF}$ which intersect the closed
subschem $\bar{X}\setminus X$ , we have an open subscheme $\PX$
that parametrize stable pairs $(\FF,s)$ with $\FF$ supported on
$X$ and let $\hat{j}:\PX\rightarrow\PXX$ be the inclusion. Let $\bar{\FFF}$
be the universal sheaf on $\PXX\times\bar{X}$ and $\bar{\ss}:\OO_{\PXX\times\bar{X}}\rightarrow\bar{\FFF}$
be the universal section, then their restriction $\FFF$, $\ss$ to
$\PX\times X$ is the universal sheaf and the universal section corresponding
to the moduli space $\PX.$ Notice that $\left(\id_{\PX}\times j\right)_{*}\FFF=\left(\hat{j}\times\id_{\bar{X}}\right)^{*}\bar{\FFF}$
on $\PX\times\bar{X}$. We also use $\FFF$ to denote $\left(\id_{\PX}\times j\right)_{*}\FFF$
on $\PX\times\bar{X}$.

There exists an action of $G=\CS$ on $\bar{X}$ by scaling the fiber
such that $X$ is an invariant open subscheme. It follows that there
exist a canonical action of $G$ on $\PXX$ and on $\PX$ . Since
$X$ is an invariant open subscheme, $\PX$ is also invariant in $\PXX$.
Thus $\bar{\FFF}$ and $\FFF$ are equivariant sheaves and $\bar{\ss}$
and $\ss$ are equivariant morphism of sheaves.

Consider the following diagrams\begin{equation}
\begin{tikzcd}[column sep=tiny]
&\PX\times X\ar[ld,"p"']\ar[rd,"q"]&\\
\PX&&X.
\end{tikzcd}
\end{equation}Let $\II$ be the complex $[\xymatrix{\OO_{\PX\times X}\ar[r]^{\,\,\,\,\,\,\,\,\,\,\,\,\,\,\,\,\,\,\,\,\,\,\,\ss} & \FFF]}
$ in $D(\PX\times X)$.  Let $\EEE^{\bullet}$ be the complex 
\[
Rp_{*}\left(R\hhom\left(\II,\II\right)_{0}\otimes\omega_{X}\right)[2].
\]
Maulik, Pandharipande and Thomas has shown that the above complex
define a perfect obstruction theory on $\PX$ \cite{MPT:10}.

Notice that $\omega_{X}\simeq\OO_{X}\otimes\tweight^{*}$ so that
by Serre's duality we have an isomorphism $\left(\EEE^{\bullet}\right)^{\vee}\rightarrow\EEE^{\bullet}[-1]\otimes\tweight$
and $\EEE$ is a symmetric equivariant obstruction theory.

Let $\PS$ be the scheme parameterizing stable pairs $\left(\FF,s\right)$
on $S$ such that the support $C_{\FF}$ of $\FF$ is in class $\beta$
and $\FF$ has Euler characteristic $\chi(\FF)=\chi$. On $\PS\times S$
there exists a universal sheaves $\FFF$ and universal section $\ss$.
With the closed embedding $\hat{i}:=\id_{\PS}\times i$$:\PS\times S\rightarrow\PS\times X$
,
\[
\xymatrix{\OO_{\PS\times X}\ar[r] & \,}
\xymatrix{\hat{i}_{*}\OO_{\PS\times S}\ar[r]^{\,\,\,\,\,\,\,\,\,\,\,\,\,\hat{i}_{*}\ss} & \hat{i}_{*}\FFF}
\]
 is a family of pairs over $\PS$. This family induces a closed embedding
$\PS\rightarrow\PX$. Indeed, $\PS$ is a connected component of $\PX^{G}$.

Let $\II_{S}$ denote the complex $[\OO_{\PS\times S}\rightarrow\FFF]$
and $\II$ denotes the complex $[\OO_{\PS\times X}\rightarrow\hat{i}_{*}\FFF]$.
Proposition 3.4 of \cite{KT:14} gives us the decomposition of $\left.\EEE\right|_{\PS}$
into its fixed and moving part as follows:
\begin{equation}
\left(\left.\EEE^{\bullet}\right|_{\PS}\right)^{fix}\simeq R\hat{p}_{*}R\hhom\left(\II_{S},\FFF\right)^{\vee}\qquad\left(\left.\EEE^{\bullet}\right|_{\PS}\right)^{mov}\simeq R\hat{p}_{*}R\hhom\left(\II_{S},\FFF\right)[1]\otimes\tweight^{*}\label{eq:decomposition}
\end{equation}
We will use $\EE^{\bullet}$ to denote $\left(\left.\EEE^{\bullet}\right|_{\PS}\right)^{fix}$.
$\left(\left.\EEE^{\bullet}\right|_{\PS}\right)^{fix}$ defines a
perfect obstruction theory on $\PS$ with virtual dimension $v=\beta^{2}+n$. 

If there is a deformation of $S$ such that the class $\beta$ is
no longer algebraic, then the virtual fundamental class will be zero
because the the virtual class is deformation invariant. If we restrict
the deformation inside the locus when $\beta$ is always algebraic
we get the reduced obstruction theory. In \cite{KT:14}, Kool and
Thomas also construct a reduded obstruction theory on $\PX$ of virtual
dimension $h^{0,2}(S)$.

Proposition 3.4 of \cite{KT:14} gives us the decomposition of $\left.\EEE_{red}^{\bullet}\right|_{\PS}$
into fixed part and moving part as follows:

\begin{align*}
\left(\left(\left.\EEE_{red}^{\bullet}\right|_{\PS}\right)^{fix}\right)^{\vee} & =\text{Cone}\left(\xymatrix{R\hat{p}_{*}R\hhom\left(\II_{S},\FFF\right)\ar[r]^{\psi\,\,\,\,\,\,\,\,} & H^{2}(\OO_{S})\otimes\OO_{\PS}[-1]}
\right)\\
\left(\left.\EEE_{red}^{\bullet}\right|_{\PS}\right)^{mov} & =R\hat{p}_{*}R\hhom(\II_{S},\FFF)[1]\otimes\tweight
\end{align*}
where $\psi$ is the composition 
\[
\xymatrix{R\hat{p}_{*}R\hhom\left(\II_{S},\FFF\right)\ar[r] & R\hat{p}_{*}R\hhom(\FFF,\FFF)[1]\ar[r]^{\text{\,\,\,\,\,\,\,\,\,\,\,\,\,\,\,\,\,\,\,\,tr}} & R\hat{p}_{*}\OO[1]\ar[r] & R^{2}\hat{p}_{*}\OO[-1]}
\]
We will use $\EE_{red}^{\bullet}$ to denote $\left(\left.\EEE_{red}^{\bullet}\right|_{\PS}\right)^{fix}$.
$\EE_{red}^{\bullet}$ defines a perfect obstruction theory on $\PS$
of virtual dimension $v_{red}=\beta^{2}+n+h^{0,2}(S)$.

For a cohomology classes $\sigma_{i}\in H^{*}(X,\ZZ)$, $i=1,\dots,m$
Kool and Thomas assign a class $\tau(\sigma_{i}):=p_{*}\left(ch_{2}\left(\FFF\right)q^{*}\sigma\right)\in H^{*}\left(\PX\right)$
where $ch_{2}\left(\FFF\right)$ is the second Chern character of
$\FFF$ and define the reduced invariants as 
\[
\mathcal{P}_{\beta,\chi}^{red}(X,\sigma_{1},\ldots,\sigma_{m}):=\int_{[\PX^{G}]^{vir}}\frac{1}{e\left(N^{vir}\right)}\prod_{i=1}^{m}\tau(\sigma_{i}).
\]
\foreignlanguage{american}{Assume that $b_{1}(S)=0$ so that $\hilb_{\beta}=|\LL|$.
It was shown that if for all $i$ , $\sigma_{i}$ is the pullback
of the Poincar\'e dual of the $[pt]\in H^{4}(S,\ZZ)$ represented
by a closed point then 
\[
\mathcal{P}_{\beta,\chi}^{red}\left(X,[pt]^{m}\right)=\int_{j^{!}[\PX^{G}]^{vir}}\frac{1}{e\left(N^{vir}\right)}
\]
where $j^{!}$ is the refined Gysin homomorphim corresponding to the
following cartesian diagram 
\[
\xymatrix{\PP^{\epsilon}\times_{|\LL|}\PX\ar[d]\ar[r] & \PX\ar[d]^{\div}\\
\PP^{\epsilon}\ar[r]^{j} & |\LL|
}
\]
where $j$ is a regular embedding $\PP^{\epsilon}\subset|\LL|$ of
a sublinear system and $\epsilon=\dim|\LL|-m$.}
\selectlanguage{american}%

\subsection{$\delta$-nodal Curve Counting via Kool-Thomas invariants}

\selectlanguage{english}%
Recall that a line bundle $\LL$ on a surface $S$ is $n$-very ample
if for any subscheme $Z$ with length $\leq n+1$ the natural morphsim
$H^{0}(X,\LL)\rightarrow H^{0}(Z,\left.\LL\right|_{Z})$ is surjective.

\selectlanguage{american}%
We assume that $b_{1}(S)=0$ and let $\LL$ be $\left(2\delta+1\right)$-very
ample line bundle on $S$ with $H^{1}(\LL)=0$. We also assume that
the first Chern class $c_{1}(\LL)=\beta\in H^{2}(S,\ZZ)$ of $\LL$
satisfies the condition that the the morphism $\cup\beta:H^{1}(T_{S})\rightarrow H^{2}(\OO_{S})$
is surjective; in particular then $H^{2}(\LL)=0$ also. Given a curve
$C$ not necessarily reduced and connected, we let $g(C)$ to denote
its arithmetic genus, defined by $1-g(C):=\chi(\OO_{C})$. If $C$
is reduced its geometric genus $\bar{g}(C)$ is defined to be the
$g(\bar{C})$ the genus of its normalisation. And let $h$ denote
the arithmetic genus of curves in $\left|\LL\right|$, so that $2h-2=\beta^{2}-c_{1}(S)\beta$.

\selectlanguage{english}%
Proposition 2.1 of \cite{KST:11} and Proposition 5.1 of \cite{KT:14}
tells us that the general $\delta$-dimensional linear system $\PP^{\delta}\subset|\LL|$
only contains reduced and irreducible curves. Moreover $\PP^{\delta}$
contains finitely many $\delta$-nodal curves with geometric genus
$h-\delta$ and other curves has geometric genus $>h-\delta.$

Kool and Thomas also define

\[
\mathcal{P}_{\chi,\beta}^{red}(S,[pt]^{m}):=\int_{\left[\PS\right]^{red}}\frac{1}{e\left(N^{vir}\right)}\tau\left(\left[pt\right]\right)^{m}.
\]
They compute $\Pslmn$ in \cite{KT2:14} and $\Pslmn$ is given by
the following expression
\begin{equation}
t^{n+\chi(\LL)-\chi(\OO_{S})}\left(-\frac{1}{t}\right)^{n+\chi(\LL)-1-m}\int\limits _{S^{[n]}\times\PP^{\chi(\LL)-1-m}}c_{n}(\LL^{[n]}(1))\frac{c_{\bullet}(T_{S^{[n]}})c_{\bullet}\left(\OO(1)^{\oplus\chi(\LL)}\right)}{c_{\bullet}\left(\LL^{[n]}(1)\right),}\label{eq:Kool-Thomas}
\end{equation}
 where $\LL^{[n]}$ is the vector bundle of rank $n$ on $S^{[n]}$
with fiber $H^{0}(\LL|_{Z})$ for a point $Z\in S^{[n]}$ and $\LL^{[n]}(1)=\LL^{[n]}\boxtimes\OO(1)$.

Under the above assumption, only the contribution from $\PS$ counts
for $\mathcal{P}_{\beta,\chi}^{red}\left(X,[pt]^{m}\right)$ so $\mathcal{P}_{\beta,\chi}^{red}\left(X,[pt]^{m}\right)=\mathcal{P}_{\chi,\beta}^{red}(S,[pt]^{m})$.
Define the generating function for $\mathcal{P}_{\beta,\chi}^{red}(X,[pt]^{m})$
as 
\[
\sum_{\chi\in\ZZ}\mathcal{P}_{\beta,\chi}^{red}(X,[pt]^{m})q^{\chi}
\]
 then define $\bar{q}=q^{1-i}(1+q)^{2i-2}$ then the coefficient of
$\bar{q}^{h-\delta}$ is $n_{\delta}(\LL)t^{h-\delta-1+\int_{\beta}c_{1}(S)}$
where $n_{\delta}(\LL)$ is the number of $\delta$-nodal curves in
$\PP^{\delta}.$

$n_{\delta}(\LL)$ has been studied for example in \cite{Go:98} and
\cite{KST:11}. In \cite{KST:11}, it is shown that after the same
change of variable $n_{\delta}(\LL)$ can be computed as the coefficient
of $\bar{q}^{h-\delta}$ of the generating function 
\[
\sum_{i=0}^{\infty}e(\hilb^{n}(\CC/\PP^{\delta}))q^{i+1-h}
\]
where $e(\hilb^{i}(\CC/\PP^{\delta})$ is the Euler characteristic
of the relative Hilbert scheme of points. Moreover $e(\hilb^{n}(\CC/\PP^{\delta}))$
can be computed as 
\[
\int_{S^{[n]}\times\PP^{\delta}}c_{i}(\LL^{[n]}(1))\frac{c_{\bullet}\left(T_{S^{[n]}}\right)c_{\bullet}\left(\OO(1)^{\oplus\delta+1}\right)}{c_{\bullet}\left(\LL^{[n]}(1)\right)}.
\]
In \cite{KST:11}, we have to assume that $\LL$ is sufficiently ample
and $H^{i}(\LL)=0$ for $i>0$ so that $\hilb^{n}(\CC/\PP^{\delta})$
are smooth. While in \cite{KT:14}, $\mathcal{P}_{\chi,\beta}^{red}(S,[pt]^{m})$
can be defined under the assumption that $H^{2}(\LL)=0$ for all $\LL$
with $c_{1}(\LL)=0$. We can think $n_{\delta}(\LL)$ as a generalization
of the one studied in \cite{KST:11}. In particular, we can think
$n_{\delta}(\LL)$ as a virtual count of $\delta$-nodal curves for
not necessarily ample line bundle $\LL$.

\section{Equivariant $K$-theoretic PT invariants of local surfaces}

In this section we will recall the $K$-theoretic invariants proposed
by Neklrasof and Okounkov in \cite{NO:14} and introduce a class that
will account for the incidence of the supporting curve of a stable
pairs and a point. The definition of this class is motivated by the
definition of points insertions in \cite{KT:14}.

Given a perfect obstruction theory $\phi:E^{\bullet}\rightarrow\LLL_{Y}$
the $K$-theoretic version of the virtual class is given in \cite{FG:10}
as follows:

\[
\OO_{Y}^{vir}:=\sum_{i}^{\infty}(-1)^{i}[\tor_{\OO_{E_{1}}}^{i}(\OO_{Y},\OO_{D})]_{Y}\in G(Y)
\]
 where $D$ is the cone\foreignlanguage{american}{ $D\subset E_{1}$
that gives the virtual class $[Y]^{vir}.$ We call $\OO_{Y}^{vir}$
the virtual structure sheaf of $Y$. Note that $\OO_{Y}^{vir}$ is
not a sheaf but a class in the Grothendieck group of coherent sheaves
on $Y$. If $\phi$ is an equivariant perfect deformation theory,
$D$ is an invariant subscheme of $E_{1}$ and we can construct $\OO_{Y}^{vir}\in G^{G}(Y)$.
If $Y$ is proper over $\CCC$, the virtual fundamental class and
virtual structure sheaf are related by the following virtual Riemann-Roch
formula by Fantechi and G\"ottsche in \cite{FG:10}}

\selectlanguage{american}%
\begin{equation}
\chi(\OO_{Y}^{vir})=\int_{[Y]^{vir}}\td(T^{vir})\label{eq:virtual-rr}
\end{equation}
where $T_{Y}^{vir}:=\left[E_{0}\right]-\left[E_{1}\right]\in K(Y)$
. We call $T_{Y}^{vir}$ the virtual tangent bundle and the dual of
it's determinant $K_{Y,vir}:=\left(\det E_{0}\right)^{-1}\otimes\det E_{1}$
$=\det E^{0}\otimes\left(\det E^{-1}\right)^{-1}\in\text{Pic}(Y)$
the virtual canonical bundle.

If $vd=0$, by equation (\ref{eq:virtual-rr}) we have 
\begin{equation}
\chi(\OO_{Y}^{vir})=\int_{[Y]^{vir}}1\in\ZZ\label{eq:virtualinvariant}
\end{equation}
so that we can use either virtual structure sheaf or virtual fundamental
class to define a numerical invariant. If there exist an isomorphism
$\theta:E^{\bullet}\rightarrow\left(E^{\bullet}\right)^{\vee}[1]$
then $\rk E^{\bullet}=\rk\left(\left(E^{\bullet}\right)^{\vee}[1]\right)=-\rk E^{\bullet}$
so that $vd=0$.

One advantage of working equivariantly is that to compute $\chi\left(\hat{\OO}_{Y}^{vir}\right)$,
we can use the virtual localization formula for the Grothendieck group
of coherent sheaves from \cite{Fen:18} by Qu . Let $G=\CS$ act on
$Y$ and $\phi:\EEE^{\bullet}\rightarrow\LLL_{Y}$ be an equivariant
perfect obstruction theory. Similar to the virtual localization formula
by Graber and Pandaripandhe, it states that, the virtual structure
sheaf equals a class coming from the fixed locus. On $Y^{G}$ we can
decompose $\EEE^{\bullet}$ into $\left(\EEE^{\bullet}\right)^{fix}\oplus\left(\EEE^{\bullet}\right)^{mov}$
where $\left(\EEE^{\bullet}\right)^{fix}$ is a two term complex with
zero weight and $\left(\EEE^{\bullet}\right)^{mov}$ is a two term
complex with non zero weight. Let $i:Y^{G}\rightarrow Y$ be the closed
embedding and let $N^{vir}=((\EEE^{\bullet})^{mov})^{\vee}.$ Then
the virtual localization formula can be stated as

\begin{equation}
i_{*}\left(\frac{\OO_{Y^{G}}^{vir}}{\bigwedge^{\bullet}\left(N^{vir}\right)^{\vee}}\right)=\OO_{Y}^{vir}\qquad\in G^{G}(Y)\otimes_{\ZZ[\tweight,\tweight^{-1}]}\QQ(\tweight)\label{eq:vir-loc}
\end{equation}
where for a two term complex $F^{\bullet}=[F^{-1}\rightarrow F^{0}]$,
$\bigwedge^{\bullet}F^{\bullet}=\frac{\sum_{i=0}^{r_{0}}(-1)^{i}\bigwedge^{i}F^{0}}{\sum_{j=0}^{r_{1}}(-1)^{j}\bigwedge^{j}F^{-1}}$
with $r_{i}=\rk F^{-i}$. On the fixed locus, the Grothendieck group
of coherent sheaves is isomorphic to the tensor product $G(Y^{G})\otimes_{\ZZ}K^{G}(pt)$
which is easier to work with.

In \cite{NO:14}, Nekrasov and Okounkov propose that we should choose
a square root of $K^{vir}$ and work with the twisted virtual structure
sheaf \cite{Qu:72}
\[
\hat{\OO}_{Y}^{vir}:=K_{Y,vir}^{\half}\otimes\OO_{Y}^{vir}.
\]
To get a refinement of (\ref{eq:virtualinvariant}), we have to consider
the action of the symmetry group of $Y$ so that $\chi\left(\hat{\OO}_{Y}^{vir}\right)$
is a function with the equivariant parameter as variables. For example
let $Y$ be the moduli space of stable pairs on a toric 3-folds $X$
and $\left(\CS\right)^{3}$acts on $Y$. Choi, Katz and Klemm have
calculated $\chi(\hat{\OO}_{Y}^{vir})$ where $X$ is the total space
of the canonical bundle $K_{S}$ for $S=\PP^{2}$ and $S=\PP^{1}\times\PP^{1}$
in \cite{CKK:12}. They have shown that the generating function with
coefficients $\chi(\hat{\OO}_{Y}^{vir})$ calculates a refinement
of BPS invariants.

To incorporate $K_{Y,vir}^{\half}$ in our computation we will consider
a double cover $G'$ of $G$ so that $\tweight^{\half}$ is a representation
of $G'$. Explicitly let $\zeta:G':=\CS\rightarrow\CS=G$, $z\mapsto z^{2}$
be the double cover. Then $G'$ acts on $Y$ via $\zeta$ by defining
$\sigma'_{Y}:G'\times Y\rightarrow Y,$$(g',y)\mapsto\sigma_{Y}(\zeta(g'),y)$
where $\sigma:G\times Y\rightarrow Y$ is the morphism defining the
action of $G$ on $Y$. Also via $\zeta$ any $G$-equivariant sheaf
$\FF$ on $Y$ is a $G'$-equivariant sheaf by pulling back the equivariant
structure via $\zeta$. This gives an exact functor $Coh^{G}(Y)\rightarrow Coh^{G'}(Y)$
and a group homomorphism $\hat{\zeta}:G^{G}(Y)\rightarrow G^{G'}(Y)$.
Moreover $\hat{\zeta}$ is a morphism of $K^{G}(\text{pt})$-modules.
For example, the primitive representation $\tweight$ of $G$ has
weight $2$ at $G'$ module. We can take the primitive representation
of $G'$ as the canonical square root of $\tweight$ and denote it
by $\tweight^{\half}$.

Next we have to compute the restriction of $K_{Y,vir}^{\half}$ on
the fixed locus. Notice that $Y^{G'}=Y^{G}$. Assume that there exist
an isomorphism $\theta:E^{\bullet}\rightarrow\left(E^{\bullet}\right)^{\vee}[1]\otimes\tweight$.
By the argument of Richard Thomas in \cite{Tho:18}, it shows that
on $Y^{G}$, $K_{Y,vir}^{\half}$ has a canonical equivariant structure.

We decompose $\left.E^{\bullet}\right|_{Y^{G}}$ into its weight spaces
so that 
\[
\left.E^{\bullet}\right|_{Y^{\CS}}=\bigoplus_{i\in\ZZ}F^{i}\tweight^{i}
\]
where $F^{i}$ are two-term complex of non-equivariant vector bundle
which only finitely many of them are nonzero and $\tweight$ is a
representation of $G$ of weight $1$. $\det E^{\bullet}$ can be
computed as the determinant of its class in $K^{G}(Y).$The isomorphism
$\theta$ implies that $[(F^{i})^{\vee}]=[F^{-i-1}[-1]]$ in $K^{G}(Y)$.
Thus $K_{Y,vir}$ is a squre twisted by a power of $\tweight$ , explicitly
\[
K_{Y,vir}=\left(\bigotimes_{i\geq0}\det\left(F^{i}\tweight^{i}\right)\right)^{\otimes2}\tweight^{r_{0}+r_{1}+\ldots}
\]
where $r_{i}=\rk F^{i}$. Thus the canonical choice for $\left.K_{Y,vir}^{\half}\right|_{Y^{G}}$
is 
\[
\bigotimes_{i\geq0}\det\left(F^{i}\tweight^{i}\right)\otimes\tweight^{\half(r_{0}+r_{1}+\ldots)}\in K^{G}(Y^{G})\otimes_{\ZZ[\tweight,\tweight^{-1}]}\ZZ[\tweight^{\half},\tweight^{-\half}].
\]
 Recall that $N^{vir}$ is the moving part of the dual of $\left.E^{\bullet}\right|_{Y^{G}}$
so that in our case $\left(N^{vir}\right)^{\vee}=\bigoplus_{i\neq0}F^{i}\tweight^{i}$. 

After choosing a square root of $K_{Y,vir}$, and that the square
root has an equivariant structure, by equation (\ref{eq:vir-loc})
we then have 
\[
i_{*}\left(\frac{\OO_{Y^{G}}^{vir}\otimes\left.K_{Y,vir}^{\half}\right|_{Y^{G}}}{\bigwedge^{\bullet}\left(N^{vir}\right)^{\vee}}\right)=\hat{\OO}_{Y}^{vir}\qquad\in K^{G}(Y)\otimes_{\ZZ[\tweight,\tweight^{-1}]}\QQ(\tweight^{\half})
\]
 If $Y$ is compact we can apply the right derived functor  $R\Gamma$
to both sides of the above equation and we have 
\begin{equation}
R\Gamma\left(Y^{G},\frac{\OO_{Y^{G}}^{vir}\otimes\left.K_{Y,vir}^{\half}\right|_{Y^{G}}}{\bigwedge^{\bullet}\left(N^{vir}\right)^{\vee}}\right)=R\Gamma\left(Y,\hat{\OO}_{Y}^{vir}\right)\in\QQ(\tweight^{\half}).\label{eq:euler-char}
\end{equation}
Thomas has proved the above identity in without using equation (\ref{eq:vir-loc}).
Furthermore Thomas has shown that 
\[
\left.R\Gamma\left(Y^{G},\frac{\OO_{Y^{G}}^{vir}\otimes\left.K_{Y,vir}^{\half}\right|_{Y^{G}}}{\bigwedge^{\bullet}\left(N^{vir}\right)^{\vee}}\right)\right|_{\tweight=1}=\int_{[Y]^{vir}}\frac{1}{e\left(N^{vir}\right)}\in\QQ
\]

In the case that we are interested on, the moduli space $Y$ is not
compact. Thus we will use the left hand side of equation (\ref{eq:euler-char})
to define our invariants.

\subsection{Equivarinat $K$-theoretic invariants}

\sloppy Let $Y$ be the moduli space of stable pairs on the canonical
bundle $X:=Spec\left(\text{\ensuremath{\sym}}\,\omega_{S}^{\vee}\right)$
of a smooth projective surface i.e. $Y=\mathcal{P}_{\chi}(X,i_{*}\beta)$
for some $\chi\in\ZZ$ and $\beta\in H_{2}(S,\ZZ)$ where $i:S\rightarrow X$
is the zero section. We will use $\pi$ to denote the structure map
$X\rightarrow S$ of $X$ as a vector bundle over $S$. Note that
$\PX$ is a quasiprojective scheme over $\CCC$. In particular, $\PX$
is separated and of finite type. 

Let $G=\CS$ act on $X$ by scaling the fiber of $\pi$. Consider
the following diagram:

\begin{equation}
\begin{tikzcd}[column sep=small]
& \PX\times X\ar[ld,"p"']\ar[rd,"q"]&\\
\PX&& X
\end{tikzcd}
\end{equation}Since $\PX$ has an equivariant perfect obstructrion theory $\phi:\EEE^{\bullet}\rightarrow\LLL_{\PX}$
where $\EEE^{\bullet}$ is the complex $Rp_{*}\left(R\hhom\left(\II,\II\right)_{0}\otimes\omega_{p}\right)[2]$
with $\omega_{P}=q^{*}\omega_{X}$ and since $X$ is Calabi-Yau $\omega_{X}\simeq\OO\otimes\tweight^{*}$
Serre duality gives us the isomorphism
\begin{equation}
\left(\EEE^{\bullet}\right)^{\vee}\simeq\EEE^{\bullet}[-1]\otimes\tweight.\label{eq:serre duality}
\end{equation}
So that by Proposition 2.6 of \cite{Tho:18} we have an equivariant
line bundle $\left.K_{\PX,vir}^{\half}\right|_{\PX^{G}}$ on $\PX^{G}$.

We want to study how to define a class that contains the information
about the incidence between a $K$-theory class in $K^{T}(X)$ and
the class of the universal sheaf $\FFF$. From another direction we
also want to give a refinement for the Kool-Thomas invariants. In
\cite{KT:14}, Kool and Thomas take the cup product of the second
Chern character of the universal sheaf $\FFF$ with the cohomology
class of points coming from $X$. Informally we could think that as
taking the intersection between the universal supporting curve and
the points of $X$.

In this article we are exploring two approaches. In the first approach
we are trying to immitate the definition of descendent used in the
article \cite{KT:14}. In \cite{KT:14} the authors are cupping the
cohomology class coming from $X$ with the second Chern class of $\FFF$.
Since we are unfamiliar on how to define Chern classes as a $K$-theory
class, we are considering to take the class of the structure sheaf
of the supporting scheme $\OO_{\CC_{\FFF}}$ and take the tensor product
of $\OO_{\CC_{\FFF}}$ with the the class coming from $X$ through
the projection $q:\PX\times X\rightarrow X$. In the second approach
we use the $K$-theory class on $\PX\times S$ of the structure sheaf
of the divisor $\div\,\pi_{*}\FFF$ and take the tensor product of
$\OO_{\div\pi_{*}\FFF}$ with the class coming from $S$ through the
projection $q_{S}:\PX\times S\rightarrow S$.

The following proposition is an equivariant version of Proposition
2.1.0 in \cite{HL:10} which we will use to define the $K$-theory
class.
\begin{prop}
\label{prop:resolution}Let $f:Y\rightarrow T$ be a smooth projective
$G$-map of relative dimension $n$ with $G$-equivariant $f$-very
ample line bundle $\OO_{Y}(1)$. Let $\FF$ be a $G$-equivariant
sheaf flat over $T$. Then there is a resolution of $\FF$ by a bounded
complex of $G$-equivariant locally free sheaves :

\[
\xymatrix{0\ar[r] & \FF_{n}\ar[r] & \FF_{n-1}\ar[r] & \ldots\ar[r] & \FF_{0}\ar[r] & \FF}
\]
where all morphisms are $G$-equivariant such that $R^{n}f_{*}F_{\nu}$
is locally free for $\nu=0,\ldots,n$ and $R^{i}f_{*}F_{\nu}=0$ for
$i\neq n$ and $\nu=0,\ldots,n$.
\end{prop}

\begin{proof}
The equivariant structure of all sheaves constructed in the proof
of Proposition 2.1.10 in \cite{HL:10} can be defined canonically.
\end{proof}
If $\OO_{\CC_{\FFF}}$ is flat over $\PX$ then $\OO_{\CC_{\FFF}}$
define a $K$-theory class in $\PX\times X$. To push the tensor product
down to a $K$-theory class in $\PX$, we push forward $\OO_{\CC_{\FFF}}$
to $\PX\times\bar{X}$ where $\bar{X}$ is $\PP(K_{S}\oplus\OO_{S})$
the projective completion of $X$. Since $\CC_{\FFF}$ is proper relative
to $\PS$ the push forward $i_{*}\OO_{\CC_{\FFF}}$ by the open embedding
$i:\PX\times X\rightarrow\PX\times\bar{X}$ is a coherent sheaf on
$\PX\times\bar{X}$. Then Proposition \ref{prop:resolution} implies
that $\OO_{\CC_{\FFF}}$ has a resolution by a finite complex of locally
free sheaf $F^{\bullet}$ on $\PX\times\bar{X}$ so that we can take
$[\OO_{\CC_{\FFF}}]:=\sum_{i}(-1)^{i}[F^{i}]$. The class $[\OO_{\CC_{\FFF}}]$
is independent of the resolution.

In section \ref{sec:Equivariant--theory} we have described the ring
homomorphism $f^{*}:K^{G}(\bar{Y})\rightarrow K^{G}(Y)$ for any morphism
of sheaves $f:Y\rightarrow\bar{Y}$. We also described the group homomorphism
$f_{*}:K^{G}(Y)\rightarrow K^{G}(\bar{Y})$ when $f$ is the structure
morphism of a projective bundle or when $f$ is finite and $f_{*}\FF$
has a resolution by locally free sheaves.

Consider the following diagram\begin{equation}
\begin{tikzcd}[column sep=small]
& \PX\times \bar{X}\ar[ld,"\bar{p}"']\ar[rd,"\bar{q}"]&\\
\PX&& \bar{X}
\end{tikzcd}
\end{equation} Let $\bar{\pi}:\bar{X}\rightarrow S$ be the structure morphism of
$\bar{X}$ as a projective bundle over $S$. We assign for each class
$\alpha\in K^{T}(X)$ a class $\gamma\left(\alpha\right)$ in $K^{T}(\PX)$
as follows. The pullback map $\pi^{*}:K^{T}(S)\rightarrow K^{T}(X)$
is an isomorphism. Thus there exist a unique class $\hat{\alpha}\in K^{T}(S)$
such that $\pi^{*}\hat{\alpha}=\alpha$. We define $\gamma\left(\alpha\right):=\bar{p}_{*}\left(\left[\OO_{\CC_{\bar{\FFF}}}\right].\left[\bar{q}^{*}\circ\bar{\pi}^{*}\hat{\alpha}\right]\right)$.
By Proposition \ref{prop:resolution}, $\left[\OO_{\CC_{\bar{\FFF}}}\right]\in K^{T}(\PX\times\bar{X})$
and since $\bar{X}$ is smooth and projective over $\CCC$, $\bar{p}_{*}$
can be defined as the composition of $i_{*}$ and $r_{*}$ where $i$
is a regular embedding and $r$ is the structure morphism $\PP_{\PX}^{N}$$\rightarrow\PX$.
Thus the class $\gamma(\alpha)$ is well defined. In particular for
every subscheme $Z\subset X$, $\gamma(\OO_{Z})$ is an element in
$K^{T}(\PX)$.

For the second approach, $\div\,\pi_{*}\FFF$ is a Cartier divisor
on $\PX\times S$ so that we have a line bundle $\OO(\div\,\pi_{*}\FFF)$
and exact sequence 
\[
\xymatrix{0\ar[r] & \OO(-\div\,\pi_{*}\FFF)\ar[r] & \OO\ar[r] & \OO_{\div\,\pi_{*}\FFF}\ar[r] & 0}
.
\]
Thus the $K$-theory class of $\OO_{\div\,\pi_{*}\FFF}$ is $1-[\OO(-\div\,\pi_{*}\FFF)]$.

Consider the following diagram \begin{equation}\label{diag:qs}
\begin{tikzcd}[column sep=small]
&\PX\times S\ar[dl,"\hat{p}"']\ar[dr,"q_S"]&\\
\PX&&S.\\
\end{tikzcd}
\end{equation}Similar to the first approach we assign for each $\alpha\in K^{T}(X)$
the class $\bar{\gamma}(\alpha):=\hat{p}_{*}\left([\OO_{\div\pi_{*}\FFF}].q_{S}^{*}\hat{\alpha}\right).$\,

In this article we only working for the case when $\alpha$ is represented
by the class of the pullback of a closed point $s\in S$. Instead
of $\gamma\left(\pi^{*}\left[\OO_{s}\right]\right)$ we will use $\gamma\left(\left[\OO_{s}\right]\right)$
to denote this class. We also assume that $b_{1}(S)=0$ so that $\hilb_{\beta}$
is simply $|\LL|$ for a line bundle $\LL$ on $S$ with $c_{1}(\LL)=\beta$.
In this article, we want to study the following invariants 
\begin{equation}
R\Gamma\left(\mathcal{P}^{G},\frac{\OO_{\mathcal{P}^{G}}^{vir}}{\bigwedge^{\bullet}\left(N^{vir}\right)^{\vee}}\otimes K_{\mathcal{P},vir}^{\frac{1}{2}}\vert_{\mathcal{P}^{G}}\otimes\left.\prod_{i=1}^{m}\beta_{i}\right|_{\PPP^{G}}\right)\in\QQ\left(\tweight^{\half}\right)\label{eq:invariants}
\end{equation}
where $\beta_{i}$ is either $\tausi$ or $\tausssi$ with $\OO_{s_{i}}$
are the classes of the structure sheaves of closed points $s_{i}\in S$.
In a special case that we have worked out in this article in order
to make the invariant coincide with Kool-Thomas invariant when we
evaluate it at $\tweight=1$ we have to replace $\gamma(\OO_{s_{i}})$
by $\frac{\gamma(\OO_{s_{i}})}{\tweightb}$ and $\tausssi$ with $\frac{\tausssi}{\tweightb}$.
Thus we define the following invariants
\[
\pxlmn:=R\Gamma\left(\PPP^{G},\frac{\OO_{\mathcal{P}^{G}}^{vir}}{\bigwedge^{\bullet}\left(N^{vir}\right)^{\vee}}\otimes K_{\mathcal{P},vir}^{\frac{1}{2}}\vert_{\mathcal{P}^{G}}\otimes\left.\prod_{i=1}^{m}\frac{\gamma(\OO_{s_{i}})}{\tweight^{-\half}-\tweight^{\half}}\right|_{\mathcal{P}^{G}}\right)
\]
 when $\OO_{\CC_{\FFF}}$ is flat and
\[
\bpxlmn:=R\Gamma\left(\PPP^{G},\frac{\OO_{\mathcal{P}^{G}}^{vir}}{\bigwedge^{\bullet}\left(N^{vir}\right)^{\vee}}\otimes K_{\mathcal{P},vir}^{\frac{1}{2}}\vert_{\mathcal{P}^{G}}\otimes\left.\prod_{i=1}^{m}\frac{\tausssi}{\tweight^{-\half}-\tweight^{\half}}\right|_{\mathcal{P}^{G}}\right)
\]

\subsection{Vanisihing of the contribution of pairs supported on thickening of
$S$ in $X$}

In this subsection we will prove that under the assumption that all
curve that pass through all the $m$ points are reduced and irreducible,
the contribution to the invariants $\pxlmn$ and $\bpxlmn$ of curves
not supported on $S$ is zero.

Proposition 2.1 of \cite{KST:11} tells us that if $\LL$ is a $2\delta+1$-very
ample line bundle on $S$ then the $\delta$-dimensional general sublinear
system $\PP^{\delta}\subset|\LL|$ only contain reduced curves. Proposition
5.1 of \cite{KT:14} also implies that these curves are also irreducible.
Thus our assumption that all curves passing through all $m$ points
are reduced and irreducible is more likely to happen. If for all $s_{i}$,
$\OO_{s_{i}}$ are in the same class, our assumption does not depend
on a particular set of $s_{i}$ but only on the number of points.

First we work for $\pxlmn$.

Let $\bar{\pi}^{\PPP}:\PX\times\bar{X}\rightarrow\PX\times S$ be
the pullback of $\bar{\pi}$ and let $i:\CC\rightarrow\PX\times\bar{X}$
be the closed embedding of the universal curve. As the composition
of projective morphisms is projective then the composition $\bar{\pi}^{\mathcal{P}}\circ i$
is also projective. Notice the above composition equals to the composition
$\CC\rightarrow\PX\times X\rightarrow\PX\times S$ which is affine.
Thus we can conclude that $\bar{\pi}^{\mathcal{P}}\circ i$ is a finite
morphism. We denote this morphism by $\rho$.

Recall the morphism $\div:\PX\rightarrow|\LL|$ from section \ref{sec:Kool-Thomas-invariants}
that maps the stable pairs $\left(\FF,s\right)$ to the supporting
curve $C_{\FF}\in|\LL|$ of $\FF$. Let $\DD\subset|\LL|\times S$
be the universal divisor and let $\DD_{\mathcal{P}}\subset\mathcal{P}\times S$
be the family of divisors that correspond to the morphism $\div:\PX\rightarrow|\LL|$
and let $j:\DD_{\mathcal{P}}\rightarrow\mathcal{P}\times S$ be the
closed embedding. Equivalently $\DD_{\PPP}=\div^{-1}\DD$.
\selectlanguage{english}%
\begin{lem}
\label{rho}$\rho$ factors through $j$.
\end{lem}

\selectlanguage{american}%
\begin{proof}
The ideal $I$ in $\OO_{\PX\times S}$ corresponding to the divisor
$\DD_{\PPP}$ is flat over $\PX$ and $\rho$ factorize through $j$
if the composition $I\rightarrow\OO_{\PX\times S}\rightarrow\rho_{*}\OO_{\CC}$
is zero. By Nakayama's Lemma it is sufficient to check whether the
composition is zero for each $p\in\PX.$ Or equivalently, we can check
whether $\rho$ factorize through $j$ at each point $p\in\PX$ .

Let $\rho^{p}:\CC_{p}\rightarrow\{p\}\times S=S$ be the restriction
of $\rho$ to the point $p\in\PX$ and let $W\subset S$ be the scheme
theoretic support of $\rho_{*}^{p}\OO_{\CC_{p}}$. Notice that $|W|=\supp(\rho_{*}\OO_{\CC_{p}})$
is a curve. We claim that $W$ is a Cartier Divisor. We will show
that $W$ is a subscheme of $\div\,\FF=\div\,\rho_{*}\OO_{\CC_{p}}$
so that $\rho^{p}$ factorize through $j^{p}$. Let $\sigma:\OO_{S}\rightarrow\rho_{*}^{p}\OO_{\CC_{p}}$
be the morphism of sheaves corresponding to the morphism $\rho^{p}:\CC_{p}\rightarrow S$.
Then $\OO_{W}$ is the image of $\sigma$ so that we have an injection
$\OO_{W}\rightarrow\rho_{*}^{p}\OO_{\CC_{p}}\rightarrow\rho_{*}^{p}\FF_{p}.$
By Proposition 1 of \cite{FP:00} we have $\div\,\rho_{*}^{p}\FF_{p}=\div\,\OO_{W}+D$
where $D$ is some effective divisor. Since $W$ is a Cartier divisor
then $\div\,\OO_{W}=W$. So that we can conclude that $W$ is a subscheme
of $\div\,\FF$.

It remains to show that $W$ is a Cartier divisor. Let $I\subset\OO_{S}$
be the ideal sheaf of $W$. It is sufficient to show that $I_{x}$
is a free $\OO_{S,x}$-module of rank 1 for every $x\in X$. For $U=S\setminus W$,
the inclusion $I\subset\OO_{S}$ is an isomorphism so that if $x\notin W,$
$I_{x}$ is isomorphic to $\OO_{S,x}$. Since $S$ is nonsingular
$\OO_{S,x}$ is a domain so that it is sufficient to show that $I_{x}$
is generated by one element $f\in\OO_{S,x}$.

Note that the morphism $\rho:\CC_{p}\rightarrow S$ is a finite morphism
so that $\left(\rho_{*}^{p}\OO_{\CC_{p}}\right)_{x}$ is a finitely
generated $\OO_{S,x}$-module. In particular, $\left(\rho_{*}^{p}\OO_{\CC_{p}}\right)_{x}$
is a Cohen-Macaulay $\OO_{S,x}$-module. By Proposition IV.13 of \cite{Ser:00},
any prime $\textbf{\ensuremath{\mathfrak{p}}}\subset\OO_{S,x}$ such
that $\OO_{S,x}/\mathbf{\mathfrak{p}}$ is isomorphic to a submodule
of $\left(\rho_{*}^{p}\OO_{\CC_{p}}\right)_{x}$ must be generated
by a single irreducible element $g\in\OO_{S,x}.$ There are finitely
many of such $\mathfrak{p}$ and we denote them by $\mathfrak{p}_{1},\ldots,\text{\ensuremath{\mathfrak{p}_{k}}}$.
Let $g_{i}$ generate $\mathfrak{p}_{i}$. By Proposition IV.11 of
\cite{Ser:00}, $I_{x}$ is the intersection $\bigcap_{i=1}^{k}\mathbf{\mathfrak{q}}_{i}$
where $\mathfrak{q}_{i}$ is an ideal of $\OO_{S,x}$ such that $\mathfrak{p}_{i}^{n_{i}}\subset\mathfrak{\mathfrak{q}}_{i}\subset\mathbf{\mathfrak{p}}_{i}$
for some positive integer $n_{i}$. Since $\OO_{S,x}$ is a domain,
$\mathfrak{q}_{i}$ must be generated by a single element $g_{i}^{m_{i}}$
for some positive integer $m_{i}$. Thus we conclude that $I_{x}$
is generated by a single element $\prod_{i=1}^{k}g_{i}^{m_{i}}$.
\end{proof}
Let $R\subset\PX^{G}$ be a connected component different from $\PS$.
We denote the inclusion $R\subset\PX$ by $\iota$. For every $(F,s)\in R$
the supporting curve $C\subset X$ is not supported by $S$ but $F$
is supported on an infinitesimal thickening of $S$ in $X$. So we
have the following diagram where all square are Cartesian\begin{equation}\label{diagR}
\begin{tikzcd}
\CC_{R}\ar[d,"i^R"']\ar[r] & \CC\ar[d,"i"']&\\
R\times\bar{X}\ar[r,"\iota_{\bar{X}}"]\ar[d,"\bar{\pi}^R"'] & \PX\times\bar{X}\ar[r,"\bar{q}"]\ar[d,"\bar{\pi}^{\mathcal{P}}"'] & \bar{X}\ar[d,"\bar{\pi}"]\\
R\times S\ar[d,"\hat{p}^R"']\ar[r,"\iota_S"] & \PX\times S\ar[d,"\hat{p}"']\ar[r,"q_S"] & S\\
R\ar[r,"\iota"]&\PX&.
\end{tikzcd}
\end{equation}By base change formula \ref{prop:BaseChange} and projection formula
\ref{prop:=00005BProjection-Formula=00005D} we have
\begin{align}
\iota^{*}\taus & =\left(\hat{p}^{R}\circ\bar{\pi}^{R}\right)_{*}\left(\iota_{\bar{X}}^{*}[\OO_{\CC}].\iota_{\bar{X}}^{*}\bar{q}^{*}\bar{\pi}^{*}\left[\OO_{s}\right]\right)\nonumber \\
 & =\hat{p}_{x*}\left(\bar{\pi}_{*}^{R}\left[\OO_{\CC_{R}}\right].\iota_{S}^{*}q_{S}^{*}\left[\OO_{s}\right]\right).\nonumber \\
\label{eq:taus}
\end{align}

Now we restrict $\rho$ from \ref{rho} to $R\subset\PX$. By the
Lemma \ref{rho} we can write $\rho^{R}$ as the composition $j^{R}\circ\lambda^{R}$.
So now we have the following diagram

\[
\xymatrix{\CC_{R}\ar[r]^{\lambda^{R}} & \DD_{R}\ar[r]^{j^{R}} & R\times S\ar[d]_{\hat{p}^{R}}\ar[r]^{\,\,\,\,\,\,q_{S}\circ\iota_{S}} & S\\
 &  & R
}
\]

By Proposition \ref{prop:resolution} the subcategory of flat coherent
sheaves on $\DD_{R}$ satisfies all conditions in Lemma \ref{lem4}
so that by Corollary \ref{Cor1} we have a group homomorphism $\lambda_{*}^{R}:K^{G}(\CC_{R})\rightarrow K^{G}\left(\DD_{R}\right)$
that maps $\left[\FF\right]$ to $\chi\left(\lambda_{*}^{R}\FF\right)$.
By the same argument we can conclude the existence of the group homomorphism
$j_{*}^{R}:K^{G}\left(\DD_{R}\right)\rightarrow K^{G}(R\times S)$.

Recall the definition of the ring homomorphism $\kappa:K^{G}(Y)\rightarrow\lim K(Y_{l})$
from Section \ref{sec:Equivariant--theory}. Although we have not
proved that $\pi_{*}^{R}\circ i_{*}^{R}[\OO_{\CC}]=j_{*}^{R}\circ\lambda_{*}^{R}[\OO_{\CC}]$,
by Lemma \ref{lem:kappa} we still have $\kappa_{R\times S}\circ\pi_{*}^{R}\circ i_{*}^{R}=\kappa_{R\times S}\circ j_{*}^{R}\circ\lambda_{*}^{R}$.
\selectlanguage{english}%
\begin{lem}
\label{lem:kappa-1}
\begin{align*}
\kappa_{R}\left(\left.\taus\right|_{R}\right) & :=\kappa_{R}\left(\hat{p}_{*}^{R}\left(\bar{\pi}_{*}^{R}\circ i_{*}^{R}[\OO_{\CC_{R}}]\otimes\iota_{S}^{*}q_{S}^{*}\left[\OO_{s}\right]\right)\right)\\
 & \,=\kappa_{R}\left(\hat{p}_{*}^{R}\left(\left(j_{*}^{R}\circ\lambda_{*}^{R}[\OO_{\CC}]\right)\otimes\iota_{S}^{*}q_{S}^{*}\left[\OO_{s}\right]\right)\right)
\end{align*}
\end{lem}

We will use $\left.\tauss\right|_{R}$ to denote $\hat{p}_{*}^{R}\left(\left(j_{*}^{R}\circ\lambda_{*}^{R}[\OO_{\CC_{R}}]\right)\otimes\iota_{S}^{*}q_{S}^{*}\left[\OO_{s}\right]\right)$
and $\left[\OO_{\CC_{R}}\right]$ to denote $\lambda_{*}[\OO_{\CC_{R}}]$.
\begin{lem}
\begin{align*}
R\Gamma\left(R,\frac{\OO_{R}^{vir}\otimes K_{vir}^{\half}\vert_{R}}{\bigwedge^{\bullet}\left(N_{vir}^{\bullet}\right)^{\vee}}\left.\prod_{i=1}^{m}\gamma\left(\OO_{s_{i}}\right)\right|_{R}\right) & =R\Gamma\left(R,\frac{\OO_{R}^{vir}\otimes\left.K_{vir}^{\half}\right|_{R}}{\bigwedge^{\bullet}\left(N_{vir}^{\bullet}\right)^{\vee}}\left.\prod_{i=1}^{m}\taussi\right|_{R}\right)
\end{align*}
\end{lem}

\selectlanguage{american}%
\begin{proof}
The Chern character map $ch^{G}:\QQ(\tweight^{\half})\rightarrow\QQ((t))$,
$\tweight^{\half}\mapsto e^{\half t}$ where $t$ is the equivariant
first Chern class of $\tweight$ is an injection since $e^{\half t}$
is invertible in $\QQ((t))$.\foreignlanguage{english}{ By virtual
Riemann-Roch theorem of \cite{FG:10}, Lemma \ref{chern} and Lemma
\ref{lem:kappa-1} we have 
\begin{align*}
ch^{G}R\Gamma\left(R,\frac{\OO_{R}^{vir}\otimes K_{vir}^{\half}\vert_{R}}{\wedge\left(N_{vir}^{\bullet}\right)^{\vee}}\left.\prod_{i=1}^{m}\gamma\left(\OO_{s_{i}}\right)\right|_{R}\right) & =ch^{G}R\Gamma\left(R,\frac{\OO_{R}^{vir}\otimes\left.K_{vir}^{\half}\right|_{R}}{\bigwedge^{\bullet}\left(N_{vir}^{\bullet}\right)^{\vee}}\left.\prod_{i=1}^{m}\taussi\right|_{R}\right).
\end{align*}
 The injectivity of $ch^{G}:\QQ(\tweight^{\half})\rightarrow\QQ((t))$
implies the lemma.}
\end{proof}
\selectlanguage{english}%
\selectlanguage{american}%
The above lemma also holds if we replace $\left.K_{vir}^{\half}\right|_{R}$
by any class $\alpha\in K^{G}(R)$.

By the above lemma we can replace $\taus$ with $\tauss=\hat{p}_{*}\left(\rho_{*}\left[\OO_{\CC}\right].q_{S}^{*}\left[\OO_{s}\right]\right)$.
The advantage of using $\tauss$ will become clear later.
\begin{lem}
Let $\LL$ be a globally generated line bundle on $S$. Let $\dim\,|\LL|=n$
and $\DD\subset|\LL|\times S$ be the universal divisor. Then for
any point $s\in S$ the fiber product $\DD\times_{|\LL|\times S}\left(|\LL|\times\{s\}\right)$
is a hyperplane $\PP^{n-1}\subset|\LL|\times\{s\}$.
\end{lem}

\begin{proof}
Let $\LL$ be globally generated line bundle on $S$ and let $f:S\rightarrow Spec\,\CCC$
be the structure morphism. Then $S\times|\LL|=Proj\,\left(\sym f^{*}\left(f_{*}\LL\right)^{\vee}\right)$
and the canonical morphism $\xi:f^{*}f_{*}\LL\rightarrow\LL$ is surjective.
Let $\xi^{\vee}:\LL^{\vee}\rightarrow f^{*}\left(f_{*}\LL\right)^{\vee}$
be the dual of $\xi$. Let $e_{i}$ be the basis of $f_{*}\LL$ and
let $e_{i}^{\vee}\in\left(f_{*}\LL\right)^{\vee}$ defined as $e_{i}^{\vee}(e_{j})=1$
if $i=j$ and $0$ if $i\neq j$. Then $\xi^{\vee}$ sends a local
section $\psi$ of $\LL^{\vee}$ to $\xi^{\vee}(\psi):\sum_{i}a_{i}e_{i}\mapsto a_{i}\psi\left(e_{i}\right)e_{i}^{\vee}$.

Sections of $f^{*}\left(f_{*}\LL\right)^{\vee}$ are linear combinations
$v$ of $\{e_{i}^{\vee}\}$ with coefficient in $\OO_{S}$ and sections
of $\sym f^{*}(f_{*}\LL)^{\vee}$ are polynomials $P$ in $\{e_{i}^{\vee}\}$
with coefficient in $\OO_{S}$. There is a canonical graded morphism
$\phi:f^{*}(f_{*}\LL)^{\vee}\otimes\sym f^{*}(f_{*}\LL)^{\vee}(-1)\rightarrow\sym f^{*}(f_{*}\LL)^{\vee}$,
that sends $v\otimes P$ to the products of the polynomials $v.P$.
The composition of $\xi^{\vee}\otimes\id_{\sym f^{*}\left(f_{*}\LL\right)^{\vee}(-1)}$
with $\phi$ sends $\psi\otimes P$ to $\ensuremath{\xi^{\vee}(\psi).P}$.
Let $\theta$ be this composition. This composition is injective since
$\xi^{\vee}$is injective. This composition correspond to the morphism
$\sigma:\LL^{\vee}\boxtimes\OO(-1)\rightarrow\OO$ on $S\times|\LL|$
which is injective because $\theta$ is injective and $Proj$ construction
preserve injective morphism. The cokernel $\sigma$ is the structure
sheaf of the universal divisor $\DD\subset S\times|\LL|$.

For any closed point $s\in S$, we want to show that the restriction
of $\sigma$ to $|\LL|$ is still injective. In this case $\DD\times_{|\LL|\times S}\left(|\LL|\times\{s\}\right)$
is an effective divisor with ideal $\OO(-1)$ so that $\DD\times_{|\LL|\times S}\left(|\LL|\times\{s\}\right)$
is a hyperplane $\PP^{n-1}$. Since $\xi$ is surjective, its restriction
to $s$ is also surjective. Any element $\alpha\in\LL^{\vee}|_{s}$
is the restriction of a local section $\psi\in\LL^{\vee}$. Thus if
$\alpha$ is not zero there exist $\psi\in\LL^{\vee}$ such that its
restriction to $s$ is $\alpha$and $e_{i}$ such that the $\left.\psi(e_{i})\right|_{s}=\left.\psi\right|_{s}\left(\left.e_{i}\right|_{s}\right)$
is not zero. We can conclude that $\left.\xi^{\vee}\right|_{s}$ is
injective. Because $\left.\sigma\right|_{s}:\left.\psi\right|_{s}\otimes\left.P\right|_{s}\mapsto\left.\xi^{\vee}(\psi)\right|_{s}\left.P\right|_{s}$
we can conclude that $\left.\sigma\right|_{s}$ is injective.
\end{proof}
We will use $\PP_{s_{i}}^{n-1}$ to denote $\DD\times_{|\LL|\times S}\left(|\LL|\times\{s\}\right)$.
\begin{lem}
Let $c_{1}\left(\LL\right)=\beta$ and let $\PPP=\PX$. Then all squares
in the following diagram are Cartesian.

\setlength{\perspective}{2pt} 
\begin{equation}\label{cube}
\begin{tikzcd}[row sep={40,between origins}, column sep={40,between origins}]       
&[-\perspective] \DD_{\PPP}\times_{\DD}\PP^{n-1}\ar{rr}\ar{dd}[near end]{\bar{h}}\ar{dl}[near start]{\bar{j}} &[\perspective] &[-\perspective] \PP^{n-1}\vphantom{\times_{S_1}} \ar{dd}\ar{dl} \\
[-\perspective]     \PPP\times\{s\} \ar[crossing over]{rr} \ar{dd}{h} & & \left|\LL\right|\times\{s\} \\
[\perspective]       & \DD_{\PPP}   \ar{rr} \ar{dl}[near start]{j} & &  \DD\vphantom{\times_{S_1}} \ar{dl} \\
[-\perspective]     \PPP\times S \ar{rr} && \left|\LL\right|\times S \ar[from=uu,crossing over] 
\end{tikzcd}
\end{equation}
\end{lem}

\begin{lem}
\label{lemintersection}If $\beta\in G^{T}(\PPP)$ is supported on
$V\subset\PPP$ then $\beta.\tauss$ is supported on $V\times_{\PPP}W_{s}$
where $W_{s}:=\DD_{\PPP}\times_{\PPP\times S}\left(\PPP\times\{s\}\right)$.
\end{lem}

\begin{proof}
Recall the morphism $\hat{p}$ from diagram (\ref{diagR}) and $h,\bar{h}$
from (\ref{cube}). Since $\hat{p}\circ h=\id_{\PPP}$ we can conclude
that $\text{\ensuremath{\tauss}}=h^{*}j_{*}\left[\OO_{\CC}\right]=h^{*}\left[j_{*}\OO_{\CC}\right]$.
Let $E^{\bullet}$ be a finite resolution of $j_{*}\OO_{\CC}$ by
locally free sheaves. It's sufficient prove the statement for the
case when $\beta$ is the class of a coherent sheaf $\FF$ on $V$
. By Lemma \ref{lem:support0}, we have
\begin{align*}
[\FF].\tauss & =\bar{j}_{*}k_{*}j^{[\OO_{\CC}]}(\FF)
\end{align*}
where $j^{[\OO_{\CC}]}$ is the refined Gysin homomorphism and $k$
is the closed embedding $V\times_{\PPP\times\{s\}}W_{s}\rightarrow W_{s}$
where $W_{s}=\DD_{\PPP}\times_{\DD}\PP_{s}^{n-1}$.
\end{proof}
\begin{lem}
\label{lem:Given-m-points}Given $m$ points $s_{1},\ldots,s_{m}\in S$
in general position such that all curves in $|\LL|$ that passes through
all $m$ points are reduced and irreducible, then for any component
$R\subset\PPP^{G}$ different from $\PS$ we have $\iota_{*}\OO_{R}^{vir}.\prod_{i=1}^{m}\taussi=0$.
\end{lem}

\begin{proof}
Let $\beta_{l}=\iota_{*}\OO_{R}^{vir}.\prod_{i=1}^{l}\hat{\gamma}\left(\OO_{s_{i}}\right)$.
By Lemma \ref{lemintersection}, $\beta_{1}$ is supported on $R\times_{\PPP}W_{s}=R\times_{|\LL|}\PP_{s_{1}}^{n-1}$.
Our assumptions implies that for any $1\leq l\leq m$, $\bigcap_{i=1}^{l-1}\PP_{s_{1}}^{n-1}$
is not contained in $\PP_{s_{l}}^{n-1}$. In particular, $\bigcap_{i=l}^{l}\PP_{s_{l}}^{n-1}=\PP^{n-m}$
and by induction we can conclude that $\beta_{m}$ is supported on
$R\times_{|\LL|}\PP^{n-m}$. Note that all curves in $\PP^{n-m}$
is reduced and irreducible.

We will show that for any $\left(\FF,s\right)\in R$, $\div\,\left(\FF,s\right)$
is not in $\PP^{n-m}$. Let $C_{\FF}$ be the curve on $X$ supporting
an element $(\FF,s)\in R$. Note that the reduced subscheme $C_{\FF}^{red}$
of $C_{\FF}$ is a curve on $S$ so that if $C_{\FF}$ is reduced
and irreducible then $C_{\FF}=C_{\FF}^{red}$ is a curve on $S$ and
$\left(\FF,s\right)$ can't be in $R$. If $C_{\FF}$ is not irreducible,
then the support of $\pi_{*}\OO_{C_{\FF}}$ is not irreduble so that
$\div\,(\FF,s)$ is not in $\PP^{n-m}$. So we are left with the case
when $C_{\FF}$ is irreducible. Let $C$ be the reduced subscheme
of $C_{\FF}$. Let $Spec\,A\subset S$ be an open subset such that
$K_{S}$ is a free line bundle over $Spec\,A$. We can write $C=Spec\,A/(f)$
for an irreducible element $f\in A$ and $X|_{Spec\,A}=Spec\,A[x]$.
Then $\OO_{C_{\FF}}$ can be written as $M:=\oplus_{i=0}^{r}A/(f^{n_{i}})x^{i}$
for some positive integers $r,n_{i}$ and $\div\,M$ is described
by the ideal $(f^{\sum_{i}n_{i}})$. Since $C_{\FF}$ is not supported
on $S$, then $\sum_{i}n_{i}\geq2$ and $\div\,M$ is not reduced.
Thus in this case $\div\,(\FF,s)$ is not in $\PP^{n-m}$.

Since $\div\,(R)$ is disjoint from $\PP^{n-m}$, we can conclude
that $R\times_{|\LL|}\PP^{n-m}$ is empty. By lemma \ref{lem:support1},
$\beta_{m}$ is zero.
\end{proof}
Following the proof of Lemma \ref{lemintersection} and Lemma \ref{lem:Given-m-points}
and by replacing $[\OO_{\CC}]$ with $[\OO_{\DD}]$ we can prove that
the contribution to $\bpxlmn$ of the component $R\subset\PPP^{G}$
where $R\neq\PS$ is zero when $s_{1},\ldots,s_{m}$ is in general
position and all curves on $S$ that passthrough all $m$ points are
reduced and irreducible.

Actually we have a stronger result for $\bpxlmn$. By Proposition
\ref{prop:ps1} for any point $s\in S$, $\tausss$ is $1-[\div^{*}\OO(-1)]$.
In particular it's independent from the choosen point.
\selectlanguage{english}%
\begin{prop}
\label{prop:assumption}Given a positive integer $\delta$, let $S$
be a smooth projective surface with $b_{1}(S)=0$. Let $\LL$ be a
$2\delta+1$-very ample line bundle on $S$ with $c_{1}(\LL)=\beta$
and $H^{i}(\LL)=0$ for $i>0$. Let $X=K_{S}$ be the canonical line
bundle over $S$. Then for any connected component $R$ of $\PX^{\CS}$
different from $\PS$ and for $m\geq H^{0}(\LL)-1-\delta$, we have
\[
R\Gamma\left(R,\frac{\OO_{R}^{vir}}{\bigwedge^{\bullet}\left(N_{vir}^{\bullet}\right)^{\vee}}K_{vir}^{\half}\vert_{R}\otimes\prod_{i=1}^{m}\frac{\tausssi}{\tweightb}\right)=0
\]
where $s_{1},\ldots s_{m}$ are closed points of $S$ which can be
identical. We then can conclude that 
\[
\bpxlmn=R\Gamma\left(\PS,\frac{\OO_{\PS}^{vir}}{\bigwedge^{\bullet}\left(N_{vir}^{\bullet}\right)^{\vee}}K_{vir}^{\half}\vert_{\PS}\otimes\prod_{i=1}^{m}\frac{\tausssi}{\tweightb}\right).
\]

The same result also holds for $\pxlmn$ under additional assumption
that the structure sheaf $\OO_{\CC_{\FFF}}$ of the universal supporting
curve $\CC_{\FFF}$ is flat over $\PX$ and $s_{1},\ldots,s_{m}$
are closed points in $S$ in general position such that all curves
in $|\LL|$ passing through all the given $m$ points are irreducible.
\end{prop}

\selectlanguage{american}%

\subsection{The contribution of $\protect\PS$}

\selectlanguage{english}%
The component $\PS$ of $\PX^{G}$ parametrize stable pairs $(F,s)$
supported on $S\subset X$ where $S$ is the zero section. The restriction
of $\II$ to $\PS\times X$ is $\II_{X}:=\{\OO_{\PS\times X}\rightarrow\FF\}$,
where $\FF$ is the universal sheaf restricted to $\PS\times X$,
so that the restriction of $\EEE^{\bullet}$ to $\PS$ is $Rp_{*}R\hhom\left(\II_{X},\II_{X}\otimes\tweight^{*}\right)_{0}[2]$
. Thomas and Kool showed that on $\PS$, the decomposition of $\left.\EEE^{\bullet}\right|_{\PS}$
into fixed and moving part is 
\begin{equation}
\left(\EEE^{\bullet}\right)^{mov}\simeq Rp_{*}R\hhom\left(\II_{S},\FF\right)[1]\otimes\tweight^{*}\qquad\left(\EEE^{\bullet}\right)^{fix}\simeq\left(Rp_{*}R\hhom\left(\II_{S},\FF\right)\right)^{\vee}\label{eq:13}
\end{equation}
where $\II_{S}=\{\OO_{\PS\times S}\rightarrow\FF\}$. $\left(\EEE^{\bullet}\right)^{fix}$
gives $\PS$ a perfect obstruction theory. We will use $\EE^{\bullet}$
to denote $\left(\EEE^{\bullet}\right)^{fix}$. From equation (\ref{eq:13})
and (\ref{eq:serre duality}) we have $\left(\EEE^{\bullet}\right)^{mov}\simeq\left(\EE^{\bullet}\right)^{\vee}[1]\otimes\tweight^{*}$.
\begin{prop}
\label{K^vir}On $\PS$ we have 
\[
\frac{K_{vir}^{\half}\vert_{\PS}}{\bigwedge^{\bullet}\left(N_{vir}^{\bullet}\right)^{\vee}}=\left(-\tweight^{-\half}\right)^{v}{\textstyle \bigwedge_{-\tweight}}\EE^{\bullet}
\]
where $vd=\rk\EE^{\bullet}$ and $\bigwedge_{-\tweight}\EE^{\bullet}=\frac{\sum_{i=0}^{\rk E^{0}}\left(-\tweight\right)^{i}\bigwedge^{i}\EE^{0}}{\sum_{j=0}^{\rk E^{-1}}\left(-\tweight\right)^{j}\bigwedge^{j}\EE^{-1}}$
for $\EE^{\bullet}=[\EE^{-1}\rightarrow\EE^{0}]$.
\end{prop}

\begin{proof}
By equation (\ref{eq:13}) and (\ref{eq:serre duality}) we have 
\begin{align*}
K_{vir}\vert_{\PS} & =\det\EE^{\bullet}\det\left(\left(\EE^{\bullet}\right)^{\vee}\otimes\tweight^{*}\right)^{\vee}=\det\EE^{\bullet}\det\EE^{\bullet}\tweight^{v}
\end{align*}
 where $v=rk$$\EE^{\bullet}.$ Thus we can take $K_{vir}^{\half}\vert_{\PS}=\det\EE^{\bullet}\tweight^{\half v}.$
Let $\EE^{\bullet}=[\EE^{-1}\rightarrow\EE^{0}]$ so that $\left(\EE^{\bullet}\right)^{\vee}[1]\otimes\tweight^{*}=[\left(\EE^{0}\right)^{\vee}\otimes\tweight^{*}\rightarrow\left(\EE^{-1}\right)^{\vee}\otimes\tweight^{*}]$
in the place of $-1$ and $0$. Let $r_{i}=rk\EE^{i}$ for $i=-1$
and $i=0$. Thus in $K^{G}(\PS)$ we have 
\begin{align*}
\frac{K_{vir}^{\half}\vert_{\PS}}{\bigwedge^{\bullet}\left(N_{vir}^{\bullet}\right)^{\vee}} & =\left(-\tweight^{-\half}\right)^{vd}{\textstyle \bigwedge_{-\tweight}}\EE^{\bullet}
\end{align*}
\end{proof}
The calculation of the contribution from this component is given in
the next section. We recall Corollary \ref{cor:Computation} here.

Under the assumption of Proposition \ref{prop:assumption} the formula
for $\bpxlmn$ is

\begin{multline*}
\left(-1\right)^{vd}\int\limits _{\left[\PS\right]^{red}}\frac{X_{-\tweight}\left(T\SN\right)X_{-\tweight}\left(\OO(1)\right)^{\delta+1}}{X_{-\tweight}\left(\ee\right)}\left(\frac{\tweightb e^{-H\left(\tweightb\right)}}{\tweightb}\right)^{m}H^{m}
\end{multline*}
where $vd$ is the virtual dimension of $\PS$ and $\OO(1)$ is the
dual of the pullback by the morphism $\div:\PX\rightarrow|\LL|$ of
the tautological line bundle and $H=c_{1}(\OO(1))$ and for any vector
bundle $E$ of rank $r$ with Chern roots $x_{1},\ldots,x_{r}$, 
\[
X_{-\tweight}(E)=\prod_{i=1}^{r}\frac{x_{i}\left(\tweightb e^{-x_{i}\left(\tweightb\right)}\right)}{1-e^{-x_{i}\left(\tweightb\right)}}.
\]
We have the same formula for $\pxlmn$ whenever $\pxlmn$ can be defined.
This is because the restriction of $\tausi$ and $\tausssi$ to $\PS$
are identical.

\selectlanguage{american}%
We can observe from the above formula that $\pxlmn$ is independent
from the choosen points. It's natural to ask if without assuming that
$s_{1},\ldots,s_{m}$ are in general positions such that all curves
passing through all these points are reduced and irreducible the above
proposition still holds.
\selectlanguage{english}%

\section{Refinemnet of Kool-Thomas Invariants}
\selectlanguage{american}%

\subsection{Reduced obstruction theory of moduli space of stable pairs on surface}

In Section 4 we have reviewed the construction of reduced obstruction
theory by Kool and Thomas in \cite{KT:14}. In this section we will
review the description of it's restriction to $\PS$ as a two term
complex of locally free sheaves following Appendix $A$ of \cite{KT:14}.
The Appendix is written by Martijn Kool, Richard P. Thomas and Dmitri
Panov.

\selectlanguage{english}%
Pandharipande and Thomas showed that $\PS$ is isomorphic to the relative
Hilbert scheme of points $\hilb^{n}(\CC/\hilb_{\beta}(S))$ where
$\CC\rightarrow\hilb_{\beta}(S)$ is the universal family of curves
$C$ in $S$ in class $\beta\in H_{2}(S,\ZZ)$ and $\chi=n+1-h$ where
$h$ is the arithmetic genus of $C$. Notice that for $n=1$ , $\PS=\hilb^{1}(\CC/\hilb_{\beta}(S))=\hilb_{\beta}(S)$.

We will review first the description of $\PS$ as the zero locus of
a vector bundle on a smooth scheme. We assume that $b_{1}(S)=0$ for
simplicity and also because we are only working for this case in this
article. The following construction does not need this assumption.

\selectlanguage{american}%
For $n=0$, pick a sufficiently ample line divisor $A$ on $S$ such
that $\LL(A)=\LL\otimes\OO(A)$ satisfies $H^{i}(\LL(A))=0$ for $i>0$.
Let $\gamma=\beta+[A]$. Then $\hilb_{\gamma}(S)=|\LL(A)|=\PP^{\chi(\LL(A)-1}$
has the right dimension. The map that send $C\in|\LL|$ to $C+A\in|\LL(A)|$
defines a closed embedding $\hilb_{\beta}(S)\rightarrow\hilb_{\gamma}(S)$.

\selectlanguage{english}%
Let $\DD\subset H_{\gamma}\left(S\right)\times S$ be the universal
divisor and let $\hat{p}$ and $q_{S}$ be the projections $H_{\gamma}(S)\times S\rightarrow H_{\gamma}(S)$
and $H_{\gamma}(S)\times S\rightarrow S$ respectively. Let $s_{\DD}\in H^{0}\left(\OO(\DD)\right)$
be the section defining $\DD$ and restrict it to $H_{\gamma}\left(S\right)\times A$
and consider the section
\begin{align*}
\zeta & :=s_{\DD}|_{\pi_{S}^{-1}A}\in H^{0}(H_{\gamma}(S)\times A,\OO(\DD)|_{\pi_{S}^{-1}A})=H^{0}(H_{\gamma}(S),\pi_{H*}(\OO(\DD)|_{\pi_{S}^{-1}A}))
\end{align*}
where for a point $D\in H_{\gamma}\left(S\right)$ we have $\zeta|_{D}=s_{D}|_{A}\in H^{0}(A,\LL(A))$
where $s_{D}$ is the section of $\LL(A)$ defining $D$. $s_{D}|_{A}=0$
if and only if $A\subset D$ i.e $D=A+C$ for some effective divisor
$C$ with $\OO(C)\otimes\OO(A)=\LL(A)$. Thus the zero locus of $\zeta$
is the image of the closed embedding $\hilb_{\beta}(S)\rightarrow\hilb_{\gamma}(S).$
If $H^{2}(\LL)=0$ then $F=\pi_{H*}(\OO(\DD)|_{\pi_{S}^{-1}A})$ is
a vector bundle of rank $\chi(\LL(A))-\chi(\LL)=h^{0}(\LL(A))-h^{0}(\LL)+h^{1}(\LL)$
on $\hilb_{\gamma}(S)$ since $R^{i}\pi_{H_{*}}\left(\OO(\DD\right)|_{\pi_{S}^{-1}A})=0$
for $i>0$. Consider the following diagram

\[
\xymatrix{\fred= & \{F^{*}\,\,\,\,\,\,\ar[r]^{d\circ\zeta^{*}\,\,\,\,\,\,\,\,\,\,\,\,\,\,\,\,\,\,\,\,\,}\ar[d]_{\zeta^{*}} & \Omega_{H_{\gamma}(S)}|_{H_{\beta}(S)}\}\ar[d]^{\text{id}}\\
\LLL_{H_{\beta}(S)}= & \{I/I^{2}|_{H_{\beta}(S)}\ar[r]^{d\,\,\,\,\,\,\,\,\,\,\,\,\,\,} & \Omega_{H_{\gamma}(S)}|_{H_{\beta}(S)}\}.
}
\]
The above morphism is a perfect obstruction theory for $\hilb_{\beta}(S)$.

Next, we embed $\hilb^{n}(\CC/\hilb_{\beta}(S))$ into $S^{[n]}\times\hilb_{\beta}(S)$.
Let $\ZZZ\subset S^{[n]}\times\hilb_{\beta}(S)\times S$ be the pullback
of the universal length $n$ subscheme of $S^{[n]}\times S$. Let
$\CC\subset S^{[n]}\times\hilb_{\beta}(S)\times S$ be the pullback
of the universal divisor of $\hilb_{\beta}\times S$ and let $\pi:S^{[n]}\times\hilb_{\beta}(S)\times S\rightarrow S^{[n]}\times\hilb_{\beta}(S)$
be the projection. Then $\CC$ correspond to a section $s_{\CC}$
of the line bundle $\OO(\CC)$ on $S^{[n]}\times\hilb_{\beta}(S)\times S.$
A point $(Z,C)\in S^{[n]}\times\hilb_{\beta}(S)$ is in the image
of $\hilb^{n}(\CC/\hilb_{\beta}(S))$ if $Z\subset C.$ We denote
by $\OO(\CC)^{[n]}$ the vector bundle $\pi_{*}\left(\OO(\CC)|_{\ZZZ}\right)$
of rank $n$. Let $\sigma_{\CC}$ be the pushforward of $s_{\CC}$
so that $\sigma_{\CC}|_{(Z,C)}=s_{C}|_{Z}\in H^{0}(\LL|_{Z})$. Thus
a point $(Z,C)\in S^{[n]}\times\hilb_{\beta}(S)$ is in the image
of $\hilb^{n}(\CC/\hilb_{\beta}(S))$ if and only if $\sigma_{\CC}|_{(Z,C)}=s_{C}|_{Z}=0$.
Thus we get a perfect relative obstruction theory :

\[
\xymatrix{\,\,\,\,\,\,\,\,\,\,\,\,\,\,\,\,\,\,\,\,\,\,\,\,\,\,\,\,\,\,\,\,\,\,\,\,\,\,E^{\bullet}= & \{\left(\OO(\CC)^{[n]}\right){}^{*}\ar[d]_{s^{*}}\ar[r]^{\,\,\,\,\,\,\,\,\,\,d\circ s^{*}} & \Omega_{\SN}\ar[d]^{\text{id}}\}\\
\LLL_{\hilb^{n}\left(\CC/\hilb_{\beta}(S)\right)/\hilb_{\beta}(S)}= & \{J/J^{2}\ar[r]^{d} & \Omega_{\SN}\}
}
\]
where $J$ is the ideal describing $\hilb^{n}\left(\CC/\hilb_{\beta}(S)\right)$
as a subscheme of $\SN\times\hilb_{\beta}(S).$ Notice that in general
$|\LL|$ is not of the right dimension.

\selectlanguage{american}%
Appendix A of \cite{KT:14} shows how to combine the above obstruction
theories to define an absolute perfect obstruction theory for $\hilb^{n}\left(\CC/\hilb_{\beta}(S)\right)$.
To do it we have to consider the embedding of $\hilb^{n}\left(\CC/\hilb_{\beta}(S)\right)$
into $S^{[n]}\times\hilb_{\gamma}(S)$. $E^{\bullet}$ is the restriction
of $[(\OO(\DD-A)^{[n]})^{*}\rightarrow\Omega_{S^{[n]}}]$ to $\hilb^{n}\left(\CC/\hilb_{\beta}(S)\right)$.
It was shown that the complex $E_{red}^{\bullet}$ that correspond
to the combined obstruction theory sits in the following exact triangle
\[
\xymatrix{F_{red}^{\bullet}\ar[r] & E_{red}^{\bullet}\ar[r] & E^{\bullet}}
.
\]
Also in Appendix A of \cite{KT:14}, it was shown that the combination
of the above obstruction theory have the same $K$-theory class with
the reduced obstruction theory $\EE_{red}^{\bullet}$. Thus we can
conclude that the $K$-theory class of $\EE_{red}^{\bullet}$ is 
\begin{equation}
[\Omega_{S^{[n]}\times\hilb_{\gamma}(S)}]-[(\OO(\DD-A)^{[n]})^{*}]-\left[F^{*}\right]\label{eq:K-class}
\end{equation}
Moreover, Theorem A.7 of \cite{KT:14} gives the virtual class corresponding
to the reduced obstruction theory $\left[\PS\right]^{red}$ as the
class $c_{n}\left(\OO(\DD-A^{[n]}\right).c_{top}\left(F\right)\cap[S^{[n]}\times\hilb_{\gamma}(S)]$.

\subsection{Point insertion and linear subsystem}

\selectlanguage{english}%
In this section we assume that $h^{0,1}(S)=0$ i.e. $\pic_{\beta}=\left\{ \LL\right\} $
and $\hilb_{\beta}(S)=|\LL|$. 

Let $\DD\subset S\times|\LL|$ be the universal curve. Pandharipande
and Thomas showed in \cite{PT2:09} that $\PS$ is isomorphic to the
relative Hilbert scheme of points $\hilb^{n}(\DD\rightarrow|\LL|)$.
There is an embedding of $\hilb^{n}(\DD\rightarrow|\LL|)$ into $S^{[n]}\times|\LL|$
and the projection $\hilb^{n}(\DD\rightarrow|\LL|)\rightarrow|\LL|$
gives a morphism $\div:\PS\rightarrow|\LL|$ that maps $(\FF,s)\in\PS$
to the supporting curve $C_{\FF}\in|\LL|$ of $\FF$.

Fix $\chi\in\ZZ$ and let $\CC$ be the universal curve supporting
the universal sheaf $\FF$ on $S\times\PS.$ Consider the following
diagram

\[
\xymatrix{\PS\times S\ar[r]^{\,\,\,\,\,\,\,\,\,\,\,\,\,q_{S}}\ar[d]_{\hat{p}} & S\\
\PS
}
\]
Of course when $n=1$ , $\PS$ is $|\LL|$ and $\CC=\DD$.

Here we will compute explicitly the class $\taus$ restricted to $\PS\rightarrow\PX^{G}$.
Note that $G$ acts trivially on $S$ and on $\PS$. Let $\CC\subset\PS\times\bar{X}$
be the support of the universal sheaf. Note that $\CC$ is supported
on $\PS\times S$ where $S$ is the zero section of the bundle $X\rightarrow S$.
Thus $\bar{\pi\circ i:\CC\rightarrow\PS\times S}$ is a closed embedding.
By equation (\ref{eq:taus}), $\taus=\hat{p}_{*}\left([\OO_{\CC}]\otimes q_{S}^{*}\left[\OO_{s}\right]\right).$
Notice that $G$ acts on $\OO_{s}$ and $\OO_{\CC}$ trivially.
\begin{prop}
\label{prop:ps0}Let $s\in S$ be a point with structure sheaf $\OO_{s}$.
Let $[\OO_{s}]$ be its class in $K(S).$ Then 
\[
\hat{p}_{*}\left(\left[\OO_{\CC}\right].q_{S}^{*}\left[\OO_{s}\right]\right)=1-[\div^{*}\OO(-1)].
\]
 where $\OO(-1)$ is the tautological line bundle on $|\LL|$.
\end{prop}

\begin{proof}
First consider the following diagram

\[
\xymatrix{|\LL|\times S\ar[r]^{\,\,\,\,\,\,\,q_{S}}\ar[d]^{\hat{p}^{|\LL|}} & S\\
|\LL|
}
.
\]
We will show that $\hat{p}_{*}\left(q_{S}^{*}[\OO_{z}].[\OO_{\DD}]\right)=1-\left[\OO(-1)\right]$.
Since $q_{S}$ is a flat morphism $q_{S}^{*}[\OO_{z}]=[q_{S}^{*}\OO_{z}]=k_{*}\left[\OO_{|\LL|\times\left\{ z\right\} }\right]$
where $k$ is the inclusion $k:|\LL|\times\{z\}\rightarrow|\LL|\times S$.
$\CC$ is the universal divisor with $\LL^{*}\boxtimes\OO(-1)$ as
the defining ideal. By the projection formula $q_{S}^{*}[\OO_{s}].[\OO_{\DD}]$
is equal to
\[
k_{*}\left[\OO_{|\LL|\times\left\{ s\right\} }\right].\left(1-\left[\LL^{*}\boxtimes\OO(-1)\right]\right)=k_{*}\left([k^{*}\OO_{|\LL|\times S}]-\left[k^{*}q_{S}^{*}\LL^{*}\otimes k^{*}\hat{p}^{*}\OO(-1)\right]\right).
\]
$k^{*}q_{S}^{*}\LL^{*}=q_{s}^{*}\LL^{*}|_{s}=\OO_{|\LL|\times\{s\}}$
where $q_{s}=q_{S}|_{|\LL|\times\{s\}}$ and $k^{*}\hat{p}^{*}\OO(-1)=\OO(-1)$
since $\hat{p}\circ k$ is the identity morphism. Thus we conclude
that 
\[
\hat{p}_{*}\left(q_{S}^{*}[\OO_{s}].[\OO_{\DD}]\right)=\hat{p}_{*}k_{*}\left(\left[\OO_{|\LL|\times\{s\}}\right]-\left[\OO(-1)\right]\right)=1-[\OO(-1)]
\]

Now we are working on $\PS.$ Consider the following Cartesian diagram
\[
\xymatrix{\div^{-1}\DD\ar[r]\ar[d] & \DD\ar[d]\\
\PS\times S\ar[r]^{\,\,\,\,\,\,\,\,\,\,\div}\ar[d]_{\hat{p}^{\PS}} & |\LL|\times S\ar[d]_{\hat{p}^{|\LL|}}\\
\PS\ar[r]^{\,\,\,\,\,\,\,\,\,\div} & |\LL|.
}
\]
$\div^{-1}\DD$ is the family of effective Cartier divisor corresponding
to the morphism $\div:\PS\rightarrow|\LL|,$ For each point $p\in\PS,$
$\div^{-1}\DD|_{p}$ is the corresponding curve $\CC_{\FF_{p}}$ supporting
the sheaf $\FF_{p}$. We conclude that $\CC$ and $\div^{-1}\CC$
are the same families of divisors on $S$ so that we have a short
exact sequence

\[
\xymatrix{0\ar[r] & \div^{*}(\LL^{*}\boxtimes\OO(-1))\ar[r] & \OO_{\PS\times S}\ar[r] & \OO_{\CC}\ar[r] & 0}
\]
 and $[\OO_{\CC}]=\div^{*}[\OO_{\DD}]$. Thus we have
\begin{align*}
\hat{p}_{*}^{\PS}\left(\left[\OO_{\CC}\right]q_{S}^{*}\left[\OO_{s}\right]\right) & =\hat{p}_{*}^{\PS}\left(\div^{*}\left[\OO_{\DD}\right].\div^{*}q_{S}^{*}\left[\OO_{s}\right]\right)\\
 & =\div^{*}\hat{p}_{*}^{|\LL|}\left(\left[\OO_{\CC}\right].q_{S}^{*}\left[\OO_{s}\right]\right)\\
 & =\div^{*}\left(1-\left[\OO(-1)\right]\right)
\end{align*}
\end{proof}
\selectlanguage{american}%
We also have similar result for $\PX$ if we replace $\OO_{\CC}$
with $\OO_{\div\,\pi_{*}\FF}$.
\begin{prop}
\label{prop:ps1}Let $\OO_{s}$ be the structure sheaf of the points
$s\in S$. Then $\hat{p}([\OO_{\div\pi_{*}\FF}].q_{S}^{*}[\OO_{s}])=1-\div^{*}(\OO(-1))$
where $\OO(-1)$ is the tautological bundle of $|\LL|$ and $\hat{p}$
, $q_{S}$ are morphism from diagram \ref{diag:qs}.
\end{prop}

\begin{proof}
From the definition of the morphism $\div$, $\div\,\pi_{*}\FF$ is
exactly $\div^{-1}\DD$. Thus we can use exactly the same proof as
the previous Proposition.
\end{proof}
\selectlanguage{english}%
Later we will drop $\div^{*}$from $\div^{*}\OO(-1)$ for simplicity.
\selectlanguage{american}%

\subsection{Refinement of Kool-Thomas invariants}

Assume that $b_{1}(S)=0$. From Proposition \ref{prop:ps0} and Proposition
\ref{prop:ps1} , the contribution of $\PS$ to $\pxlmn$ and to $\bpxlmn$
are equal. Consider the contribution of $\PS$ to $\bpxlmn$ invariants,
i.e.

\[
\Xi=R\Gamma\left(\PS,\frac{\OO_{\PS}^{vir}\otimes K_{vir}^{\half}}{\bigwedge^{\bullet}(N^{vir})^{\vee}}\prod_{i=1}^{m}\frac{\tausssi}{\tweightb}\right).
\]
 \foreignlanguage{english}{On $\hilb_{\beta}(S)\times S$ we have
the following exact sequence 
\begin{equation}
\xymatrix{0\ar[r] & \OO\ar[r]^{s_{\CC}} & \OO(\CC)\ar[r] & \OO_{\CC}(\CC)\ar[r] & 0}
\label{eq:reduced}
\end{equation}
which induces the exact sequence 
\[
\xymatrix{H^{1}(\OO_{\CC}(\CC))\ar[r]^{\hat{\phi}} & H^{2}(\OO_{S})\ar[r] & H^{2}(\LL)}
.
\]
If $H^{2}(\LL)=0$ then $\hat{\phi}$ is surjective. Observe that
$R\pi_{H*}\OO_{\CC}(\CC)$ is the complex $\EE^{\bullet}$ from Subsection
2.2.1 when $\chi=2-h$ or equivalently when $n=1$. For $n>1$, it
was shown in Appendix A of \cite{KT:14} that $\EE^{\bullet}$ sits
in the exact triangle
\[
\xymatrix{R\pi_{H*}\OO_{\CC}(\CC)\ar[r] & \EE^{\bullet}\ar[r] & E^{\bullet}}
.
\]
Thus if $h^{2}(\OO_{S})>0$ then $\EE^{\bullet}$ contain a trivial
bundle so that $[\PS]^{vir}$ vanish. In particular, by virtual Riemann-Roch
the contribution of $\PS$ is zero.}

If $H^{2}(\OO_{S})=0$, $\EE_{red}^{\bullet}$ and $\EE^{\bullet}$
are quasi isomorphic. Let $P$\foreignlanguage{english}{ be the moduli
space $\PS.$ By the virtual Riemann-Roch theorem and by Lemma \ref{prop:ps1}
we then have}

\selectlanguage{english}%
\begin{align*}
ch^{G}\left(\Xi\right) & =\left(-\tweight^{-\half}\right)^{vd}\int_{[P]^{red}}\ch\left({\textstyle \bigwedge_{-\tweight}}\eered\left(\frac{{\textstyle \bigwedge_{-1}}\OO(-1)}{\tweightb}\right)^{m}\right).\td\left(T_{P}^{red}\right)
\end{align*}
where $T_{P}^{red}$ is the derived dual of $\eered$ and $\left(-\tweight^{-\half}\right)^{vd}$
should be understood as $\left(-e^{-\half t}\right)^{vd}$ where $t$
is the equivariant first Chern class of $\tweight$. Observe that
$ch^{G}\left(\Xi\right)$ can be computed whenever $H^{2}(\LL)=0$
without assuming $h^{2}(\OO_{S})=0$. Thus for $S$ with $b_{1}(S)=0$
and a line bundle $\LL$ with $H^{2}(\LL)=0$, we define $\pslmn=ch^{G}(\Xi)$.

The $K$-theory class of $\EE_{red}^{\bullet}$ is given by equation
(\ref{eq:K-class}). Since $\OO(\CC)=\LL\boxtimes\OO(1)$, by the
projection formula we have $F=H^{0}(\LL(A)\rvert_{A})\otimes\OO(1)$.
From the exact sequence 
\[
\xymatrix{0\ar[r] & \OO(\CC)\ar[r] & \OO(\CC+A)\ar[r] & \OO_{\pi_{S}^{-1}A}(\CC+\pi_{S}^{-1}A)\ar[r] & 0}
\]
on $P$, and since $H^{i>0}(\LL(A)\rvert_{A})=0$, we conclude that
\begin{align}
F & =\OO(1)^{\oplus\chi(\LL(A))-\chi(\LL)}\label{eq:4}
\end{align}
And again by projection formula we have $\OO(\CC)^{[n]}=\LL^{[n]}\boxtimes\OO(1)$
. By Theorem A.7 of \cite{KT:14} we then can compute $\pslmn$ as
\begin{multline}
\tweightc^{v}\int\limits _{S^{[n]}\times|\LL(A)|}H^{\chi(\LL(A))-\chi(\LL)}c_{n}(\OO(\DD-A)^{[n]})\\
\ch\left(\frac{{\textstyle \bigwedge_{-\tweight}}\eered\left(\bigwedge_{-1}\OO(-1)\right)^{m}}{\left(\tweightb\right)^{m}}\right)\td\left(T_{P}^{red}\right)\label{eq:5}
\end{multline}
where $H=c_{1}(\OO(1)$ and $n=\chi+h-1$.
\begin{thm}
\label{prop:evaluation}$\pslmn\rvert_{\tweight=1}=(-1)^{vd}\int_{\SN\times\PP^{\varepsilon}}c_{n}(\LL^{[n]}\otimes\OO(1))\frac{c_{\bullet}\left(T\SN\right)c_{\bullet}(\OO(1))^{\chi(\LL)}}{c_{\bullet}\left(\LL^{\{n]}\boxtimes\OO(1)\right)}$
where $\varepsilon=\chi(\LL)-1-m$. Thus we can relate Kool-Thomas
invariants with our invariants as follows:
\begin{eqnarray*}
\mathcal{P}_{\chi,\beta}^{red}\left(S,[pt]^{m}\right) & = & \left(-1\right)^{m}t^{m+1-\chi(\OO_{S})}\left.\pslmn\right|_{\tweight=1}.
\end{eqnarray*}
\end{thm}

\begin{proof}
Let $\XX_{-\tweight}(T_{P}^{red}):=\ch\left(\bigwedge_{-\tweight}\eered\right)\td\left(T_{P}^{red}\right)$
and let $d:=\rk\EE_{red}^{\bullet}=n+\chi(\LL)-1$ be the virtual
dimension of $P$ so that we can rewrite (\ref{eq:5}) as 
\begin{equation}
\left(-1\right)^{m}\tweightc^{d-m}\int_{\SN\times\PP^{\chi(\LL)-1}}c_{n}\left(\LL^{[n]}\boxtimes\OO(1)\right)\XX_{-\tweight}\left(T_{P}^{red}\right)\ch\left(\frac{{\textstyle \bigwedge_{-1}}\left(\OO(-1)\right)}{\tweighta}\right)^{m}\label{eq:6}
\end{equation}
By Proposition 5.3 of \cite{FG:10} we can write 
\[
\XX_{-\tweight}(T_{P}^{red})=\sum_{l=0}^{d}\left(\tweighta\right)^{d-l}\XX^{l}
\]
where $\XX^{l}=c_{l}(T_{P}^{red})+b_{l}$where $b_{l}\in A^{>l}(P)$.
Then we can write $\pslmn$ as 
\[
\left(-1\right)^{m}\tweightc^{d-m}\int_{\SN\times\PP^{\chi(\LL)-1}}c_{n}(\LL^{[n]}\boxtimes\OO(1))\sum_{l=0}^{d}(\tweighta)^{d-m-l}\XX^{l}\ch\left({\textstyle \bigwedge_{-1}}\left(\OO(-1)\right)\right)^{m}.
\]
Note that $\ch\left({\textstyle \bigwedge_{-1}}\left(\OO(-1)\right)\right)^{m}=H^{m}+O\left(H^{m+1}\right)$
so that 
\[
\int_{\SN\times\PP^{\chi(\LL)-1}}c_{n}\left(\LL^{[n]}\boxtimes\OO(1)\right)\XX^{l}\ch\left({\textstyle \bigwedge_{-1}}\left(\OO(-1)\right)\right)^{m}=0
\]
 for $l>d-m$. Thus the summation ranges from $l=0$ to $l=d-m$.
In this range the power of $(\tweighta)$ is positive except when
$l=d-m$ in which the power of $(\tweighta)$ is zero. Thus we can
conclude that $\pslmn\rvert_{\tweight=1}$equals to 
\[
\left(-1\right)^{m}\tweightc^{d-m}\int_{\SN\times\PP^{\chi(\LL)-1}}c_{n}\left(\LL^{[n]}\boxtimes\OO(1)\right)\XX^{d-m}\ch\left({\textstyle \bigwedge_{-1}}\left(\OO(-1)\right)\right)^{m}.
\]
 Since $b_{d-m}\in A^{>d-m}(P)$ and $c_{d-m}(T_{P}^{red})\in A^{d-m}(P)$
we have
\[
\int_{\SN\times\PP^{\chi(\LL)-1}}c_{n}\left(\LL^{[n]}\boxtimes\OO(1)\right)b_{d-m}\ch\left({\textstyle \bigwedge_{-1}}\left(\OO(-1)\right)\right)^{m}=0
\]
and 
\[
\int_{\SN\times\PP^{\chi(\LL)-1}}c_{n}\left(\LL^{[n]}\boxtimes\OO(1)\right)c_{d-m}(T_{P}^{red})H^{k}=0
\]
for $k>m$ and we can conclude that 
\begin{align*}
\pslmn\rvert_{\tweight=1} & =\left(-1\right)^{\half d}\int_{\SN\times\PP^{\chi(\LL)-1}}c_{n}\left(\LL^{[n]}\boxtimes\OO(1)\right).H^{m}.c_{d-m}(T_{P}^{red})
\end{align*}
From (\ref{eq:K-class}) and (\ref{eq:4}) we have 
\[
T_{P}^{red}=T\left(\SN\right)+\OO(1)^{\chi(\LL(A))}-\OO-\LL^{[n]}\boxtimes\OO(1)-\OO(1)^{\chi(\LL(A))-\chi(\LL)}
\]
and 
\[
c_{d-m}(T_{P}^{red})=\coeff_{t^{d-m}}\left[\frac{c_{t}\left(T\SN\right)c_{t}\left(\OO(1)\right)^{\chi(\LL)}}{c_{t}\left(\LL^{[n]}\boxtimes\OO(1)\right)}\right].
\]
Finally we conclude that 
\[
\pslmn\rvert_{\tweight=1}=\left(-1\right)^{-\half d}\int_{\SN\times\PP^{\delta}}c_{n}(\LL^{[n]}\otimes\OO(1))\frac{c_{\bullet}\left(T\SN\right)c_{\bullet}(\OO(1))^{\chi(\LL)}}{c_{\bullet}\left(\LL^{\{n]}\boxtimes\OO(1)\right)}
\]
\end{proof}
Let $X_{-y}(x)\in\QQ[[x,y]]$ defined by 
\[
X_{-y}(x):=\frac{x\left(y^{-\half}-y^{\half}e^{-x(y^{-\half}-y^{\half})}\right)}{1-e^{-x\left(y^{-\half}-y^{\half}\right)}}.
\]
For a vector bundle $E$ on a scheme $Y$of rank $r$ with Chern roots
$x_{1},\ldots,x_{r}$ we will use $X_{-y}(E)$ to denote 
\[
\prod_{i=1}^{r}\frac{x_{i}\left(y^{-\half}-y^{\half}e^{-x_{i}(y^{-\half}-y^{\half})}\right)}{1-e^{-x_{i}\left(y^{-\half}-y^{\half}\right)}}.
\]
Observe that $X_{-y}$ is additive on an exact sequence of vector
bundle. Thus we can extend $X_{y}$ to $K(Y)$. For a class $\beta\in K(Y)$
we can write $\beta=\sum_{i}[E_{i}^{+}]-\sum_{j}[E_{j}^{-}]$ for
vector bundles $E_{i}^{+},E_{j}^{-}$ and we can define $X_{-y}(\beta)=\frac{\prod_{i}X_{-y}(E_{i}^{+})}{\prod_{j}X_{-y}(E_{j}^{-})}.$
For a proper nonsingular scheme $Y$ with tangent bundle $T_{Y}$
\[
\int_{Y}X_{-y}\left(T_{Y}\right)=\left(\frac{1}{y}\right)^{\half d}\sum_{i}(-1)^{p+q}y^{q}h^{p,q}(Y)
\]
 where $h^{p,q}(Y)$ are the Hodge number of $Y$ i.e. $\int_{Y}X_{-y}\left(T_{Y}\right)$
is the normalized $\chi_{-y}$ genus.
\begin{thm}
\label{prop:2}
\begin{multline}
\pslmn=\\
(-1)^{vd}\int\limits _{\left[P\right]^{red}}\frac{X_{-\tweight}\left(T\SN\right)}{X_{-\tweight}\left(\ee\right)}X_{-\tweight}\left(\OO(1)\right)^{\delta+1}\left(\frac{\tweightb e^{-H\left(\tweightb\right)}}{\tweightb}\right)^{m}H^{m}\label{eq:7}
\end{multline}
 where $[P]^{red}$ is $c_{n}\left(\LL^{[n]}\boxtimes\OO(1)\right)\cap[S^{[n]}\times\PP^{\chi(\LL)-1}]$.
\end{thm}

\selectlanguage{american}%
\begin{proof}
$\pslmn$ equals to (\ref{eq:6}), and we can rewrite it as

\selectlanguage{english}%
\[
\pslmn=\left(-1\right)^{vd}\int_{\left[P\right]^{red}}\frac{\prod_{i=1}^{2n+\chi(\LL)-1}\frac{\phi_{-\tweight}(\alpha_{i})}{\tweight^{1/2}}}{\prod_{i=1}^{n}\frac{\phi_{-\tweight}(\beta_{i})}{\tweight^{1/2}}}\left(\frac{1-e^{-H}}{\tweightb}\right)^{m}
\]

where $\phi_{-\tweight}(x)=\frac{x\left(\tweighta e^{-x}\right)}{1-e^{-x}}$
and $\alpha_{i}$ are the Chern roots of $T(S^{[n]}\times\PP^{\chi(\LL)-1})$
and $\beta_{i}$ are the Chern roots of $\LL^{[n]}\boxtimes\OO(1)$.
Let's define
\[
\bar{\phi}_{-\tweight}(x):=\frac{\phi_{-\tweight}(x)}{\tweight^{1/2}}=\frac{x\left(\tweightb e^{-x}\right)}{1-e^{-x}}=\sum_{i\geq0}\bar{\phi}_{i}x^{i}.
\]

Note that this power series starts with $\tweightb$. By substituting
$x$ with $x\left(\tweightb\right)$ and dividing it by $\left(\tweightb\right)$
we have the power series 
\[
X_{-\tweight}(x)=\frac{x\left(\tweightb e^{-x\left(\tweightb\right)}\right)}{1-e^{-x\left(\tweightb\right)}}=\sum_{i\geq0}\xi_{i}x^{i}
\]
such that $\xi_{0}=1$ and $\xi_{i}=\bar{\phi}_{i}\left(\tweightb\right)^{i-1}$.
Thus by substituting $x$ in 
\[
\frac{\prod_{i=1}^{2n+\chi(\LL)-1}\frac{\phi_{-y}(\alpha_{i})}{\tweight^{1/2}}}{\prod_{i=1}^{n}\frac{\phi_{-y}(\beta_{i})}{\tweight^{1/2}}}\left(\frac{1-e^{-H}}{\tweightb}\right)^{m}
\]
 with $x\left(\tweightb\right)$ whenever $x=\alpha_{i},\beta_{i},H$
and dividing it by 
\[
\left(\tweightb\right)^{n+\chi(\LL)-1}
\]
so that the coefficients of $q^{n+\chi(\LL)-1}$ in 
\[
\frac{\prod_{i=1}^{2n+\chi(\LL)-1}X_{-\tweight}(\alpha_{i}q)}{\prod_{i=1}^{n}X_{-\tweight}(\beta_{i}q)}\left(\frac{1-e^{-Hq\left(\tweightb\right)}}{\tweightb}\right)^{m}
\]
and 
\[
\frac{\prod_{i=1}^{2n+\chi(\LL)-1}\bar{\phi}_{-\tweight}(\alpha_{i}q)}{\prod_{i=1}^{n}\bar{\phi}_{-\tweight}(\beta_{i}q)}\left(\frac{1-e^{-Hq}}{\tweightb}\right)^{m}
\]
are the same. Since $[T\PP^{\chi(\LL)-1}]=[\oplus_{i=1}^{\chi(\LL)}\OO(1)]-[\OO_{\PP^{\chi(\LL)-1}}],$
$\pslmn$ equals 
\begin{multline*}
\left(-1\right)^{vd}\int_{[P]^{red}}\frac{X_{-\tweight}\left(T\SN\right)X_{-\tweight}\left(\OO(1)\right)^{\chi(\LL)}}{X_{-\tweight}\left(\LL^{[n]}\boxtimes\OO(1)\right)}\left(\frac{1-e^{-H\left(\tweightb\right)}}{\tweightb}\right)^{m}=\\
\left(-1\right)^{vd}\int_{[P]^{red}}\frac{X_{-\tweight}\left(T\SN\right)X_{-\tweight}\left(\OO(1)\right)^{\delta+1}}{X_{-\tweight}\left(\LL^{[n]}\boxtimes\OO(1)\right)}\left(\frac{\tweightb e^{-H\left(\tweightb\right)}}{\tweightb}\right)^{m}H^{m}
\end{multline*}
\end{proof}
\selectlanguage{english}%
In the following Corollary we want to complete the computation of
$\pxlmn$.
\begin{cor}
\label{cor:Computation}Given a positive integer $\delta$, let $S$
be a smooth projective surface with $b_{1}(S)=0$. Let $\LL$ be a
$2\delta+1$-very ample line bundle on $S$ with $c_{1}(\LL)=\beta$
and $H^{i}(\LL)=0$ for $i>0$. Let $X=K_{S}$ be the canonical line
bundle over $S$. Then for $m=\chi(\LL)-1-\delta$ points $s_{1},\ldots,s_{m}$
which is not necessarily different
\begin{multline*}
\bpxlmn=\\
\left(-1\right)^{vd}\int_{\left[P\right]^{red}}\frac{X_{-\tweight}\left(T\SN\right)X_{-\tweight}\left(\OO(1)\right)^{\delta+1}}{X_{-\tweight}\left(\ee\right)}\left(\frac{\tweightb e^{-H\left(\tweightb\right)}}{\tweightb}\right)^{m}H^{m}
\end{multline*}
where $[P]^{red}=c_{n}(\LL^{[n]}\boxtimes\OO(1))\cap[S^{[n]}\times\PP^{\chi(\LL)-1}]$
for $m\geq H^{0}(\LL)-1-\delta$.

\selectlanguage{american}%
If additionally $\OO_{\CC_{\FFF}}$ is flat over $\PX$ and $s_{1},\dots,s_{m}$
are closed points of $S$ in general position such that all curves
on $S$ that pass through all $m$ points are reduced and irredcible
then $\pxlmn$ is given by the same formula.
\end{cor}

\selectlanguage{american}%
\begin{proof}
By Proposition \ref{prop:assumption} $\bpxlmn=\pslmn$. Similarly
for $\pxlmn$.
\end{proof}
In \cite{GS:14,GS:15}, for every smooth projective surface $S$ and
line bundle $\LL$ on $S$, Göttsche and Shende defined the following
power series 
\[
D^{S,\LL}(x,y,w):=\sum_{n\geq0}w^{n}\int_{\SN}X_{-y}\left(T\SN\right)\frac{c_{n}\left(\LL^{[n]}\otimes e^{x}\right)}{X_{-y}\left(\LL^{[n]}\otimes e^{x}\right)}\in\QQ\llbracket x,y,w\rrbracket
\]
where $e^{x}$ denotes a trivial line bundle with nontrivial $\CS$
action with equivariant first Chern class $x$. Motivated by this
power series we define a generating function 
\begin{equation}
P_{S,\LL,m}:=\sum_{n\geq0}\left(-w\right)^{n}P_{S,\LL,m,n+1-h}.\label{eq:8}
\end{equation}
 where $h$ is the arithmetic genus of the curve $C$ in $S$ with
$\OO(C)\simeq\LL$ so that for the pair $(\FF,s)\in\PS$, $n=\chi-1+h$
.

By Theorem \ref{prop:2}, after substituting $\tweight$ by $y$ we
can rewrite $P_{S,\LL,m}$ as 
\[
\coeff_{x^{\delta}}\left[D^{S,\LL}(x,y,w)X_{-y}(x)^{\delta+1}\left(\frac{\yb e^{-x\left(\yb\right)}}{\yb}\right)^{m}\right]
\]
Note that 
\[
Q_{S,\LL,m}:=\coeff_{x^{\delta}}\left[D^{S,\LL}(x,y,w)X_{-y}(x)^{\delta+1}\right]
\]
is equation (2.1) of \cite{GS:15} and 
\[
\left(\frac{\yb e^{-x\left(\yb\right)}}{\yb}\right)^{m}
\]
is a power series starting with $1$. 

\selectlanguage{english}%
In \cite{GS:15}, Gottsche and Shende defined the power series $N_{\chi(\LL)-1-k,[S,\LL]}^{i}(y)$
by the following equation:
\begin{equation}
\sum_{i\in\ZZ}N_{\chi(\LL)-1-k,[S,\LL]}^{i}(y)\left(\frac{w}{(1-y^{-1/2}w)(1-y^{1/2}w)}\right)^{i+1-g}=Q_{S,\LL,m}\label{eq:3-1}
\end{equation}
Motivated by this we also define $M_{\chi(\LL)-1-m,[S,\LL]}^{i}(y)$
as 
\begin{equation}
\sum_{i\in\ZZ}M_{\chi(\LL)-1-m,[S,\LL]}^{i}(y)\left(\frac{w}{(1-y^{-1/2}w)(1-y^{1/2}w)}\right)^{i+1-g}=\bpslm\label{eq:4-1}
\end{equation}

Let's define $\frac{1}{Q}=\frac{(1-y^{-1/2}w)(1-y^{1/2}w)}{w}=w+w^{-1}-y^{-1/2}-y^{1/2}$
and recall a conjecture from \cite{GS:14}.
\begin{conjecture}
[Conjecture 55 of \cite{GS:14}]\label{conj:1}
\[
\left(\frac{w(Q)}{Q}\right)^{1-g(\LL)}D^{S,\LL}(x,y,w(Q))\in\QQ[y^{-1/2},y^{1/2}]\llbracket x,xQ\rrbracket
\]
\end{conjecture}

Motivated by the conjecture above we define another power series
\[
\tilde{D}^{S,\LL}(x,y,Q):=\left(\frac{w(Q)}{Q}\right)^{1-g(\LL)}D^{S,\LL}(x,y,w(Q)).
\]

\begin{prop}
\label{prop:Assume-conjecture}Assume \textup{Conjecture \ref{conj:1}.
For $\chi(\LL)-1\geq k\geq0$ we have }

\textup{1. $M_{\chi(\LL)-1-k,[S,\LL]}^{i}(y)=0$ and $N_{\chi(\LL)-1-k}^{i}(y)=0$
for $i>\chi(\LL)-1-k$ and for $i\leq0$. }

\textup{2. $M_{\chi(\LL)-1-k,[S,\LL]}^{i}(y)$ and $N_{\chi(\LL)-1-k}^{i}(y)$
are Laurent polynomials in $y^{1/2}$.}

\textup{3. Furthermore $M_{\chi(\LL)-1-k,[S,\LL]}^{\chi(\LL)-1-k}(y)=N_{\chi(\LL)-1-k,[S,\LL]}^{\chi(\LL)-1-k}(y)$.
Moreover 
\[
\sum_{i\ge0}M_{\delta,[S,\LL]}^{\delta}(y)\left(s\right)^{\delta}=\tilde{D}^{S,\LL}(x,y,\frac{s}{x})|_{x=0}=\sum_{\delta\geq0}N_{\delta,[S,\LL]}^{\delta}(y)s^{\delta}
\]
}
\end{prop}

\begin{proof}
After substituting $w$ by $w(Q)$ we rewrite equation (\ref{eq:3-1})
and (\ref{eq:4-1})
\[
\sum_{i\in\ZZ}N_{\delta,[S,\LL]}^{i}(y)x^{\delta-i}\left(xQ\right)^{i}=\left[\tilde{D}^{S,\LL}(x,y,Q)X_{-y}(x)^{\delta+1}\right]_{x^{\delta}}
\]
\begin{multline*}
\sum_{i\in\ZZ}M_{\delta,[S,\LL]}^{i}(y)x^{\delta-i}\left(xQ\right)^{i}=\\
\left[\tilde{D}^{S,\LL}(x,y,Q)X_{-y}(x)^{\delta+1}\left(\frac{y^{1-/2}-y^{1/2}e^{-x\left(\yb\right)}}{\yb}\right)^{m}\right]_{x^{\delta}}.
\end{multline*}
By Conjecture \ref{conj:1} 
\[
\sum_{i\in\ZZ}N_{\delta,[S,\LL]}^{i}(y)x^{\delta-i}\left(xQ\right)^{i},\sum_{i\in\ZZ}M_{\delta,[S,\LL]}^{i}(y)x^{\delta-i}\left(xQ\right)^{i}\in\QQ[y^{-1/2},y^{1/2}]\llbracket x,xQ\rrbracket
\]
 so that the only possible power of $Q$ that could appear is $i=0,\ldots,\delta$.
We can directly conclude that $N_{\delta,[S,\LL]}^{i},M_{\delta,[S,\LL]}^{i}$
are Laurent polynomial in $y^{1/2}.$ Set $s=xQ,$ so that by Conjecture
\ref{conj:1} we can write $\tilde{D}^{S,\LL}(x,y,Q)$ as power series
of $x$ and $s$ i.e $\tilde{D}^{S,\LL}(x,y,\frac{s}{x})\in\QQ[y^{-1/2},y^{1/2}]\llbracket x,s\rrbracket.$
And since 
\begin{align*}
X_{-y}(x=0) & =1\\
\left(\frac{y^{1-/2}-y^{1/2}e^{-x\left(\yb\right)}|_{x=0}}{\yb}\right)^{m} & =1
\end{align*}
we can conclude that 
\begin{align*}
\sum_{i\ge0}M_{\delta,[S,\LL]}^{\delta}(y)\left(s\right)^{\delta} & =\tilde{D}^{S,\LL}(x,y,\frac{s}{x})|_{x=0}\\
 & =\sum_{\delta\geq0}N_{\delta,[S,\LL]}^{\delta}(y)s^{\delta}
\end{align*}
\end{proof}
\selectlanguage{american}%
If $H^{i}(\LL)=0$ for $i>0$ and $\LL$ is $\delta$-very ample,
then $N_{\delta,[S,\LL]}^{\delta}(y)$ is the refinement defined by
Goettsche and Shende in \cite{GS:14} of $n_{\delta}(\LL)$ that computes
the number of $\delta$-nodal curves in $|\LL|$. Theorem \ref{prop:2}
and Theorem \ref{prop:evaluation} gives geometric argument for the
equality $M_{\delta,[S,\LL]}^{\delta}(y)|_{y=1}=N_{\delta,[S,\LL]}^{\delta}(y)|_{y=1}$.
Without assuming the conjecture above we would like to know if Proposition
\ref{prop:Assume-conjecture} still true.
\selectlanguage{english}%

\section{Acknowledgements}

This project start as my PhD thesis under the supervision of Lothar
G\"ottsche under ICTP/SISSA joint PhD program. The finisihing of
this project is supported by P3MI ITB research funds.

\global\long\def\bibname{References}%

\selectlanguage{american}%
\bibliographystyle{styles/bibtex/gillow}
\addcontentsline{toc}{section}{\refname}\bibliography{references/Bibliography}
\selectlanguage{english}%

\end{document}